\documentclass[11pt]{article}
\usepackage[margin=.9in]{geometry}

\usepackage{graphicx}
\usepackage{amssymb}
\usepackage{amsmath}
\usepackage{verbatim,booktabs}
\usepackage{subfigure,epsfig,url,psfrag,caption}
\usepackage{float}
\usepackage{mathrsfs}
\usepackage{threeparttable}
\usepackage{multirow}
\usepackage[usenames,dvipsnames]{pstricks}
\usepackage{epsfig}
\usepackage{pst-grad} 
\usepackage{pst-plot} 
\usepackage{enumitem}
\usepackage{hyperref}
\usepackage{todonotes}
\usepackage{algpseudocode}

\newcommand{\din}{d^{\text{in}}}
\newcommand{\dout}{d^{\text{out}}}

\usepackage[capitalize,noabbrev]{cleveref}

\usepackage[utf8]{inputenc}
\usepackage{fullpage}
\usepackage{amsfonts}
\usepackage{xcolor}
\usepackage{changepage}

\allowdisplaybreaks

\crefname{claim}{Claim}{Claims}
\floatname{algorithm}{Algorithm}

\usepackage{amsthm}

\newtheorem{corollary}{Corollary}
\newtheorem{proposition}{Proposition}
\newtheorem{lemma}{Lemma}

\newtheorem{conjecture}{Conjecture}

\newtheorem{claim}{Claim}

\def\st{{\rm s.t.}}

\def\R{{\mathbb R}}

\def\ie{{i.e.,} }

\def\S{{\mathcal S}}
\def\H{{\mathcal H}}

\def\N{{\mathcal N}}

\def\B{{\mathcal B}}

\def\01{\ensuremath{0\mathord{-}1}}

\newcommand{\Tr}{\text{Tr}}

\newcommand{\norm}[1]{\Big\lVert#1\Big\rVert}
\usepackage[ruled,vlined]{algorithm2e}

\renewenvironment{claim}[1][]
{\refstepcounter{claim} \begin{trivlist} \item[] {\bf Claim~\theclaim}\space#1 \itshape}
{\end{trivlist}}

\newenvironment{cpf}
{\begin{trivlist} \item[] {\em Proof of claim }}
{$\hfill\diamond$ \end{trivlist}}

\def\Xint#1{\mathchoice
{\XXint\displaystyle\textstyle{#1}}%
{\XXint\textstyle\scriptstyle{#1}}%
{\XXint\scriptstyle\scriptscriptstyle{#1}}%
{\XXint\scriptscriptstyle\scriptscriptstyle{#1}}%
\!\int}
\def\XXint#1#2#3{{\setbox0=\hbox{$#1{#2#3}{\int}$ }
\vcenter{\hbox{$#2#3$ }}\kern-.6\wd0}}

\def\dashint{\Xint-}
\def\Xsum#1{\mathchoice
{\XXsum\displaystyle\textstyle{#1}}%
{\XXsum\textstyle\scriptstyle{#1}}%
{\XXsum\scriptstyle\scriptscriptstyle{#1}}%
{\XXsum\scriptscriptstyle\scriptscriptstyle{#1}}%
\!\sum}
\def\XXsum#1#2#3{{\setbox0=\hbox{$#1{#2#3}{\sum}$ }
\vcenter{\hbox{$#2#3$ }}\kern-.5\wd0}}

\def\dashsum{\Xsum-}

\DeclareMathOperator{\prob}{\mathbb P}
\DeclareMathOperator{\avg}{\mathbb E}

 \allowdisplaybreaks

\title{On the power of linear programming for K-means clustering
\thanks{The authors were partially funded by AFOSR grant FA9550-23-1-0123.}}

\author{Antonio De Rosa
\thanks{Department of Decision Sciences and BIDSA, Bocconi University, Via R\"ontgen 1, 20136, Milano, MI, Italy. E-mail: {\tt  antonio.derosa@unibocconi.it}.
             }          
\and
Aida Khajavirad
\thanks{Department of Industrial and Systems Engineering, Lehigh University, Bethlehem, PA 18015, USA.
             E-mail: {\tt aida@lehigh.edu}.
             }
             \and
Yakun Wang
\thanks{Department of Industrial and Systems Engineering, Lehigh University, Bethlehem, PA 18015, USA.
             E-mail: {\tt yaw220@lehigh.edu}.
             }
}
\date{}

\begin{document}

\maketitle

\begin{abstract}
    In~\cite{AntoAida20}, the authors introduced a new polynomial-size linear programming (LP) relaxation for 
    K-means clustering. In this paper, we further investigate both theoretical and computational properties of this relaxation. As evident from our numerical experiments with both synthetic and real-world data sets, the proposed LP relaxation is almost always tight; \ie its optimal solution is feasible for the original nonconvex problem. To better understand this unexpected behavior, on the theoretical side, we focus on K-means clustering with two clusters, and we obtain sufficient conditions under which  a given partition of data is an optimal solution of the LP relaxation. We further analyze the sufficient conditions when the input is generated according to a popular stochastic model and obtain recovery guarantees for the LP relaxation. We conclude our theoretical study by constructing a family of inputs for which the LP relaxation is never tight. Denoting by $n$ the number of data points to be clustered, the LP relaxation contains $\Omega(n^3)$ inequalities making it impractical for large data sets.
    To address the scalability issue, by building upon a cutting-plane algorithm together with the GPU implementation of PDLP, a first-order method LP solver, we develop an efficient algorithm that solves the proposed LP and hence the K-means clustering problem, for up to $n \leq 4000$ data points. 
      
\end{abstract}


\section{Introduction}
\label{sec:intro}
Clustering data points into a small number of groups according to some similarity measure is a common task in unsupervised machine learning. K-means clustering, one of the oldest and most popular clustering techniques, partitions data points into clusters by minimizing the total squared distance between each data point and the corresponding cluster center.
Let $\{x^i\}_{i=1}^n$ denote a set of $n$ data points in $\R^m$, and denote by $K$ the number of desired clusters. Define a partition of $[n] := \{1,\ldots, n\}$ as a family $\{\Gamma_k\}_{k =1}^K$ of non-empty subsets of $[n]$ such that
$\Gamma_a \cap \Gamma_b = \emptyset$ for all $a\neq b \in [K]$ and $\cup_{k \in [K]} {\Gamma_k} = [n]$. Then K-means clustering can be formulated as a combinatorial optimization problem:
\begin{align}
\label{Kmeans1}
\min  \quad & \sum_{k=1}^K{\sum_{i \in \Gamma_k} { \norm{x^i - \frac{1}{|\Gamma_k|}\sum_{j \in \Gamma_k}{x^j}}_2^2}}\\
\text{s.t.} \quad & \{\Gamma_k\}_{k \in [K]} \; {\rm is \; a \; partition \; of \; [n]}.  \nonumber
\end{align}
It is well-known that Problem~\eqref{Kmeans1} is NP-hard even when there are only two clusters~\cite{Aloise09} or when the data points are in $\R^2$~\cite{MahNimVar09}. The most popular techniques for solving Problem~\eqref{Kmeans1} are heuristics such as Lloyd's algorithm~\cite{Lloyd82}, approximation algorithms~\cite{Kan02,FriRezSal19}, and convex relaxations~\cite{PenXia05,PenWei07,Awaetal15,MixVilWar17,IduMixPetVil17,LiLiLi20,AntoAida20}. The two prominent types of convex relaxations for K-means clustering are semi-definite programming (SDP) relaxations~\cite{PenWei07} and linear programming (LP) relaxations~\cite{AntoAida20}. The theoretical properties of SDP relaxations for K-means clustering have been thoroughly studied in the literature~\cite{Awaetal15,MixVilWar17,IduMixPetVil17,PraHan18,LiLiLi20}.
In this paper, we investigate the power of LP relaxations for K-means clustering. In the following, we present an alternative formulation for Problem~\eqref{Kmeans1}, which we use to construct our relaxations.

Consider a partition $\{\Gamma_k\}_{k =1}^K$ of $[n]$; let ${\bf 1}_{\Gamma_k}$, $k \in [K]$ be the indicator vector of
the $k$th cluster; \ie  the $i$th component of ${\bf 1}_{\Gamma_k}$ is defined as: $({\bf 1}_{\Gamma_k})_i = 1$ if $i \in \Gamma_k$ and $({\bf 1}_{\Gamma_k})_i = 0$ otherwise. Define
the associated \emph{partition matrix} as:
\begin{equation}\label{pm}
X = \sum_{k=1}^K{\frac{1}{|\Gamma_k|}{\bf 1}_{\Gamma_k}{\bf 1}^T_{\Gamma_k}}.
\end{equation}
Define $d_{ij} := ||x^i - x^j||_2^2$ for all $i,j \in [n]$.
Then Problem~\eqref{Kmeans1} can be equivalently written as (see~\cite{LiLiLi20} for detailed derivation):
\begin{align}
\label{Kmeans2}
\min \quad & \sum_{i,j \in [n]} {d_{ij} X_{ij}}\\
\text{s.t.} \quad & X \; {\rm is \; a \; partition\; matrix \; defined \; by~\eqref{pm}}.  \nonumber
\end{align}
It can be checked that any partition matrix $X$ is positive semidefinite.
Using this observation, one can obtain the following SDP relaxation of Problem~\eqref{Kmeans2}:
\begin{align}\label{SDP}
\tag{PW}
\min \quad & \sum_{i,j \in [n]} {d_{ij} X_{ij}}\\
\text{s.t.} \quad &  {\rm Tr}(X) = K,\quad \sum_{j=1}^n X_{ij} = 1, \quad \forall i \in [n],\nonumber\\
                 & X \succeq 0,\quad X_{ij} \geq 0,\quad X_{ij} = X_{ji}, \quad  \forall 1 \leq i < j \leq n \nonumber,
\end{align}
where ${\rm Tr}(X)$ is the trace of the matrix $X$, and
$X \succeq 0$ means that $X$ is positive semidefinite. The above relaxation was first proposed in~\cite{PenWei07} and is often referred to as the ``Peng-Wei SDP relaxation''. Both theoretical and numerical properties of this relaxation have been thoroughly investigated in the literature~\cite{Awaetal15,MixVilWar17,IduMixPetVil17,LiLiLi20,PicSudWie22}.
 
\subsection{LP relaxations for K-means clustering}
In~\cite{AntoAida20}, the authors introduced the \emph{ratio-cut polytope}, defined as the convex hull of ratio-cut vectors corresponding to all partitions of $n$ points in $\R^m$ into at most $K$ clusters. They showed that this polytope is closely related to the convex hull of the feasible region of Problem~\eqref{Kmeans2}. The authors then studied the facial structure of the ratio-cut polytope, which in turn enabled them to obtain a new family of LP relaxations for K-means clustering. Fix a parameter $t \in \{2,\ldots, K\}$;
then an LP relaxation for K-means clustering is given by:
\begin{align}
\label{lp:pK}
\tag{LPK$_t$}
\min  \quad  & \sum_{i,j \in [n]} {d_{ij} X_{ij}}\nonumber\\
\text{s.t.} \quad  & \Tr(X) = K,\quad \sum_{j=1}^n X_{ij} = 1, \quad \forall i \in [n],\nonumber\\
& \sum_{j \in S} {X_{ij}} \leq X_{ii} + \sum_{j, k \in S: j < k} {X_{jk}}, \quad \forall i \in [n], \; \forall S \subseteq [n] \setminus\{i\}: 2 \leq |S| \leq t,\label{eq3k}\\
& X_{ij} \geq 0, \quad X_{ij}=X_{ji}, \quad \forall 1 \leq i < j \leq n \nonumber.
\end{align}
%
By Proposition~3 of~\cite{AntoAida20}, inequalities~\eqref{eq3k} define facets of the convex hull of the feasible region of Problem~\eqref{Kmeans2}. It then follows that for any $t < t'$, the feasible region of Problem~(LPK$_{t'}$) is strictly contained in the feasible region of Problem~(LPK$_t$). However, it is important to note that Problem~\eqref{lp:pK} contains $\Theta(n^{t+1})$ inequalities; moreover, an inequality of
the form~\eqref{eq3k} with $|S| = t'$ contains $\frac{t'(t'+1)}{2}+1$ variables with nonzero coefficients. That is, while increasing $t$ improves the quality of the relaxation, it significantly increases the computational cost of solving the LP as well. Hence, a careful selection of $t$ is of crucial importance; we address this question in Section~\ref{sec:numerics}.
In fact, the numerical experiments on synthetic data in~\cite{AntoAida20}
indicate that Problem~\eqref{lp:pK} with $t=2$ is almost always \emph{tight}; \ie the optimal solution of the LP is a partition matrix.
To better understand this unexpected behavior, in this paper, we perform a theoretical study of the tightness of the \eqref{lp:pK} relaxation. As a first step, we consider the case of two clusters. Performing a similar type of analysis for $K >2$ is a topic of future research.

\vspace{0.1in}
Our main contributions are summarized as follows:
\begin{itemize}[leftmargin=*]
    \item [(i)] Consider any partition of $[n]$ and the associated partition matrix $X$
    defined by~\eqref{pm}. By constructing a dual certificate, we obtain a sufficient condition, often referred to as a \emph{proximity condition}, under which $X$ is the unique optimal solution of Problem~\eqref{lp:pK} (see Proposition~\ref{Th: optimality} and Proposition~\ref{th: uniqueness}).
    This result can be considered as a generalization of Theorem~1 in~\cite{AntoAida20} where the optimality of \emph{equal-size} clusters is studied.
    Our proximity condition is overly conservative since,
    to obtain an explicit condition, we fix a subset of the dual variables to zero. To address this, we propose a simple algorithm to carefully assign values to dual variables previously set to zero, leading to a significantly better dual certificate for $X$ (see Proposition~\ref{prop:alg}).

    \item [(ii)] Consider a generative model, referred to as the \emph{stochastic sphere model} (SSM), in which there are two clusters of possibly different size in $\R^m$, and the data points in each cluster are sampled from a uniform distribution on a sphere of unit radius. Using the result of part~(i) we obtain a sufficient condition in terms of the distance between sphere centers under which the \eqref{lp:pK} relaxation recovers the planted clusters with high probability. By high probability we mean the probability tending to one as the number of data points tends to infinity (see Proposition~\ref{recovery}).


    \item [(iii)] Since in all our numerical experiments  Problem~\eqref{lp:pK} is tight, it is natural to ask whether the \eqref{lp:pK} relaxation is tight with high probability under reasonable generative models. We present a family of inputs for which the \eqref{lp:pK} relaxation is never tight (see Proposition~\ref{non-tightness} and Corollary~\ref{rem-tight}).

    \item [(iv)] We propose a scalable cutting plane algorithm, which relies on an efficient separation algorithm for inequalities~\eqref{eq3k}, lower bounding and upper bounding techniques, and a GPU implementation of PDLP, a first-order primal-dual LP solver. Thanks to this algorithm, we are able to solve Problem~\eqref{lp:pK} for real-world instances with up to $4000$ data points in less than two and half hours. Surprisingly, for \emph{all} instances, by solving the LP relaxation, we solve the original K-means clustering problem to \emph{global optimality}.

\end{itemize}

\subsection{Organization}
The remainder of the paper is structured as follows. In Section~\ref{sec:motivate}, we perform a theoretical study of the strength of the LP relaxation for two clusters.
In Section~\ref{sec:negative}, we present a family of inputs for which the optimal value of Problem~\eqref{lp:pK} is strictly smaller than the optimal value of the K-means clustering problem. In Section~\ref{sec:numerics} we develop a customized algorithm for solving the LP relaxation and present extensive numerical experiments. Finally, in Section~\ref{sec:appendix} we present the proof of a technical lemma that is omitted from Section~\ref{sec:motivate}.

\section{The tightness of the LP relaxation for two clusters}
\label{sec:motivate}

In this section, we consider the K-means clustering problem with $K=2$ and examine the strength of the LP relaxation.  Recall that for $K=2$, we have $t=2$; thus, Problem~\eqref{lp:pK}  simplifies to:
\begin{align}
\tag{LP2}
\label{lp:pK2}
\min  \quad & \sum_{i,j \in [n]} {d_{ij} X_{ij}} \nonumber\\
\text{s.t.} \quad & \Tr(X) = 2, \label{e1k2}\\
& \sum_{j=1}^n X_{ij} = 1, \quad \forall 1\leq i\leq n, \label{e2k2}\\
& X_{ij}+X_{ik} \leq X_{ii} + X_{jk}, \quad \forall i \neq j \neq k \in [n], \; j < k, \label{e3k2}\\
& X_{ij} \geq 0, \quad \forall 1 \leq i < j \leq n, \label{e4k2}
\end{align}
where as before we set $X_{ji} = X_{ij}$ for all $1 \leq i < j \leq n$. Throughout this section, whenever we say ``the LP relaxation,'' we mean the LP defined by Problem~\eqref{lp:pK2} and whenever we say ``the SDP relaxation,'' we mean the SDP defined by Problem~\eqref{SDP}.
To motivate our theoretical study, in the following we present a simple numerical experiment to convey the power of the LP relaxation for clustering real-world data sets. Solving both LP and SDP becomes computationally expensive as we increase the number of points $n$; that is, to solve an instance with $n \gtrsim 200$ efficiently, one needs to design a specialized algorithm for both LP and SDP.
Indeed, in~\cite{PicSudWie22} the authors design a specialized algorithm for the SDP relaxation and solve problems with $n \leq 4000$. In Section~\ref{sec:numerics}, we present a customized algorithm to solve the LP relaxation. In this section, however, we are interested in performing a theoretical analysis of the tightness of the LP and we are using our numerical experiments to motivate this study. Hence, we limit ourselves to solving problem instances with $n \leq 200$ variables. 

We say that a convex relaxation is tight if its optimal solution is a partition matrix. In~\cite{PenXia05}, the authors prove that a symmetric matrix $X$ with unit row-sums and trace $K$ is a partition matrix if and only if it is a \emph{projection} matrix; \ie $X^T X = X$. Hence to check whether the solution of the LP relaxation or the SDP relaxation is tight, we check whether it is a projection matrix. 
We collected a set of 15 data sets from the UCI machine learning repository~\cite{dua2017uci}. The data sets and their characteristics, \ie the number of points $n$ and the input dimension $m$ are listed in columns 1-3 of Table~\ref{table1}. For data sets with $n \leq 200$ (\ie Voice, Iris, and Wine), we solve the clustering problem over the entire data set. For each of the remaining data sets, to control the computational cost of both LP and SDP, we sample $10$ times $n'=200$ data points and solve LP and SDP using the $10$ random instances. All experiments are performed on the {\tt NEOS} server~\cite{neos98};
LPs are solved with {\tt GAMS/Gurobi}~\cite{gurobi} and SDPs are solved with {\tt GAMS/MOSEK}~\cite{mosek}.
The tightness rate of the LP ($t_{\rm LP}$), the tightness rate of the SDP ($t_{\rm SDP}$), and the (average) relative gap between the optimal values of the SDP and the LP are listed in columns 4-6 of Table~\ref{table1}. In \emph{all} instances, the LP relaxation is tight, while the SDP is tight only in a few instances.  
The remarkable performance of the LP relaxation is indeed unexpected, and hence in this paper, it is our goal to better understand this behavior through a theoretical study. 

\begin{table}
\centering
\caption{Comparing the strength of the LP  versus the SDP for clustering real-world data with $K=2$.}
\small
\begin{tabular}{l|l|l|l|l|l}
\toprule														
Data set	& $n$ & $m$ &$t_{\rm LP}$  &  $t_{\rm SDP}$ &   $g_{\rm rel}$ (\%)\\
\midrule	
Voice   & 126 & 310 &1.0 &   0.0   & 5.40\\
Iris   & 150  & 4 & 1.0  &   0.0   & 1.09 \\
Wine    & 178  & 13  &1.0 &   0.0    & 3.44 \\
Seeds    &  210  & 7  &1.0   &   0.0  & 3.80\\
Accent &  329  &  12 & 1.0   &   0.3   & 0.02\\
ECG5000 &  500  &  140 &1.0      &  0.1  & 0.07      \\
Hungarian &  522 & 20  &1.0     &   0.0  &    1.01\\
Wdbc  &  569  & 30  &  1.0      &    0.0   &   5.51 \\
Strawberry & 613  & 235  & 1.0      &  0.0  &  5.40  \\
Energy    &  768 &  16 &  1.0     &  1.0   &  0.0    \\
SalesWeekly & 810 & 106 & 1.0     &  0.0  &   1.74  \\
Vehicle &  846 & 18  & 1.0       &   0.0  &   1.11 \\
Wafer &  1000 & 152  &  1.0        &  0.3   &   0.06 \\
Ethanol & 2000 &  27 &  1.0      &    1.0   &   0.0   \\
Rice & 3810  &  7 &   1.0     &     0.0   &   1.70\\
\bottomrule															
\end{tabular} \label{table1}
\end{table}

\subsection{A proximity condition for the LP relaxation}
\label{sec:proximity}

We first obtain a proximity condition for the LP relaxation given by Problem~\eqref{lp:pK2}.
A proximity condition for the SDP relaxation is presented in~\cite{LiLiLi20}.
As our proximity condition is overly conservative, we then present a simple algorithm that certifies the optimality of a given partition matrix.
To establish our proximity condition, in Proposition~\ref{Th: optimality}, we first obtain a sufficient condition under which a given partition matrix is an optimal solution of Problem~\eqref{lp:pK2}. Subsequently in Proposition~\ref{th: uniqueness}, we address the question of uniqueness of the optimal solution. The next proposition
can be considered as a generalization of Theorem~1 in~\cite{AntoAida20} where the authors study the optimality of \emph{equal-size} clusters. 

In the following, for every $A \subseteq [n]$ and $f:A\to \R$, we define
$$
\underset{i \in A}{\dashsum}{f(i)} := \frac 1{|A|}\sum_{i \in A}{f(i)}.
$$
Moreover, given a partition $\{\Gamma_1, \Gamma_2\}$ of $[n]$,
for each $i \in \Gamma_l$, $l \in \{1,2\}$,  we define 
$$\din_i := \underset{j\in \Gamma_l}{\dashsum}{d_{ij}}, \qquad \dout_i := \underset{j \in [n]\setminus \Gamma_l}{\dashsum}{d_{ij}}.$$

We are now ready to state our proximity condition.

\begin{proposition}\label{Th: optimality}
    Let $\{\Gamma_1, \Gamma_2\}$ denote a partition of $[n]$ and assume without loss of generality that $|\Gamma_1| \leq |\Gamma_2|$.
    Define
    \begin{equation}\label{rdef}
    r_1 := \frac{2|\Gamma_1|}{|\Gamma_1|+|\Gamma_2|}, \quad r_2 := \frac{2|\Gamma_2|}{|\Gamma_1|+|\Gamma_2|},
    \end{equation}
    and
    \begin{equation}\label{etadef}
    \eta := \frac{r_2}{2}\Bigg(\Big(1-\frac{r_1}{r_2}\Big) \max_{k \in \Gamma_1}\din_k +\Big(1-\frac{r_2}{r_1}\Big) \min_{k\in \Gamma_2}\din_k+ \frac{r_1}{r_2} \underset{k \in \Gamma_1}{\dashsum}{\din_k}+\frac{r_2}{r_1} 
   \underset{k \in \Gamma_2}{\dashsum}{\din_k}\Bigg).
    \end{equation}
    Suppose that
\begin{equation}\label{ass1}
\underset{k \in \Gamma_2}{\dashsum}{\min\{r_2 d_{ik} + \din_j , r_2 d_{jk} + \din_i\}} -d_{ij} \geq \eta, \quad \forall i < j \in \Gamma_1, 
\end{equation}    
and
\begin{equation}\label{ass2}
\underset{k \in \Gamma_1}{\dashsum}{\min\{r_1 d_{ik} + \din_j , r_1 d_{jk} + \din_i\}} -d_{ij} \geq \frac{r_1}{r_2}\eta, \quad \forall i < j \in \Gamma_2.
\end{equation}    
    Then, an optimal solution of Problem~\eqref{lp:pK2} is given by the partition matrix
    \begin{equation}\label{optsol}
    \bar X=\frac{1}{|\Gamma_1|}{\bf 1}_{\Gamma_1}{\bf 1}^T_{\Gamma_1}+\frac{1}{|\Gamma_2|}{\bf 1}_{\Gamma_2}{\bf 1}^T_{\Gamma_2}.
    \end{equation}
\end{proposition}
\begin{proof}
 We start by constructing the dual of Problem~\eqref{lp:pK2}. Define dual variables $\omega$ associated with~\eqref{e1k2}, $\mu_i$ for all $i \in [n]$ associated  with~\eqref{e2k2},
$\lambda_{ijk}$ for all $(i,j,k) \in \Omega := \{(i,j,k): i \neq j \neq k \in [n], \; j < k\}$ associated  with~\eqref{e3k2}, and $\sigma_{ij}$ for all $1 \leq i < j \leq n$ associated with~\eqref{e4k2}. The dual of Problem~\eqref{lp:pK2} is then given by:
\begin{align}
\max \quad  & -(2\omega + \sum_{i \in [n]}{ \mu_i}) \nonumber\\
\text{s.t.} \quad & \mu_i + \mu_j + \sum_{k \in [n] \setminus \{i,j\}} {(\lambda_{ijk} + \lambda_{jik}-\lambda_{kij})} + 2 d_{ij} - \sigma_{ij}= 0, \;
\forall 1 \leq i < j \leq n, \label{d1}\\
&\omega + \mu_i -\sum_{j,k \in [n] \setminus \{i\}: j < k}{\lambda_{ijk}} = 0, \; \forall i \in [n],\label{d2}\\
& \lambda_{ijk} \geq 0, \quad \forall (i,j,k) \in \Omega, \quad \sigma_{ij} \geq 0, \quad \forall 1 \leq i < j \leq n, \nonumber
\end{align}
where we let $\lambda_{ikj} =\lambda_{ijk}$ for all $(i,j,k) \in \Omega$. 
To establish the optimality of $\bar X$ defined by~\eqref{optsol}, it suffices to construct a dual feasible point $(\bar \lambda, \bar \mu, \bar \omega, \bar \sigma)$ that together with $\bar X$ satisfies the complementary slackness. 
Without loss of generality, assume that $i < j$ for all $i \in \Gamma_1$ and $j \in \Gamma_2$.  Then $\bar X$ and $(\bar \lambda, \bar \mu, \bar \omega, \bar \sigma)$ satisfy the complementary slackness if and only if:
\begin{itemize}
\item [(i)] $\bar\lambda_{ijk} = 0$ if $i \in \Gamma_1$ and $j,k \in \Gamma_2$ or if $i \in \Gamma_2$ and $j,k \in \Gamma_1$, 
\item [(ii)] $\bar \sigma_{ij} = 0$ if $i,j \in \Gamma_1$ or $i, j \in \Gamma_2$. 
\end{itemize}
After projecting out $\bar \sigma_{ij}$ for $i \in \Gamma_1$ and $j \in \Gamma_2$, we deduce that it suffices to find $(\bar \lambda, \bar \mu, \bar \omega)$ satisfying:
\begin{align}
    & \bar \mu_i + \bar \mu_j + \sum_{k \notin \Gamma_l} {(\bar\lambda_{ijk} + \bar\lambda_{jik})}
+ \sum_{k \in \Gamma_l \setminus \{i,j\}} {(\bar\lambda_{ijk} + \bar\lambda_{jik}-\bar\lambda_{kij})}
+ 2 d_{ij} = 0, \quad \forall i < j \in \Gamma_l, l \in \{1,2\}\label{c1}\\
&\bar\mu_i + \bar\mu_j + \sum_{k \in \Gamma_1 \setminus \{i\}} {(\bar\lambda_{ijk}-\bar\lambda_{kij})} + \sum_{k \in \Gamma_2 \setminus \{j\}} {(\bar\lambda_{jik}-\bar\lambda_{kij})} + 2 d_{ij} \geq 0, \quad \forall i \in \Gamma_1, j \in \Gamma_2 \label{c2}\\
&\bar\omega + \bar\mu_i -\sum_{\substack{j \in \Gamma_l \setminus \{i\},\\ k \notin \Gamma_l}}{\bar\lambda_{ijk}} -\sum_{j <k \in \Gamma_l \setminus \{i\}}{\bar\lambda_{ijk}}= 0, \quad \forall i \in \Gamma_l, l \in \{1,2\}\label{c3}. 
\end{align}
To this end, let 
\begin{align}
&\bar \lambda_{ijk} - \bar \lambda_{jik}= \frac{d_{jk}-d_{ik}}{n/2}+\frac{\din_{i}-\din_{j}}{|\Gamma_2|}, \quad \forall i, j \in \Gamma_1, k \in \Gamma_2 \label{g1}\\
&\bar \lambda_{ijk} - \bar \lambda_{jik}= \frac{d_{jk}-d_{ik}}{n/2}+\frac{\din_{i}-\din_{j}}{|\Gamma_1|}, \quad \forall i, j \in \Gamma_2, k \in \Gamma_1 \label{g2}\\
    & \bar \mu_i = -\din_i - r_2 \dout_i + \eta, \quad \forall i \in \Gamma_1 \label{g3}\\
    & \bar \mu_i = -\din_i - r_1 \dout_i +\frac{r_1}{r_2}\eta, \quad \forall i \in \Gamma_2 \label{g4}\\
    & \bar \omega = -\frac{1}{2}\sum_{i \in [n]} {(\din_i+\bar \mu_i)},\label{g5}
\end{align}
where $r_1, r_2$ and $\eta$ are as defined by~\eqref{rdef} and~\eqref{etadef}, respectively.

First let us examine the validity of inequalities~\eqref{c2}. 
Substituting~\eqref{g1}-\eqref{g4} in~\eqref{c2} yields:
\begin{align*}
   & -\din_i - r_2 \dout_i + \eta -\din_j - r_1 \dout_j +\frac{r_1}{r_2}\eta
    + \sum_{k \in \Gamma_1 \setminus \{i\}}{\Big(\frac{d_{jk}-d_{ij}}{n/2}+\frac{\din_{i}-\din_{k}}{|\Gamma_2|}\Big)}\\
   & + \sum_{k \in \Gamma_2 \setminus \{j\}}{\Big(\frac{d_{ik}-d_{ij}}{n/2}+\frac{\din_{j}-\din_{k}}{|\Gamma_1|}\Big)} + 2 d_{ij} = \\
   & -\din_i - r_2 \dout_i + \eta -\din_j - r_1 \dout_j +\frac{r_1}{r_2}\eta + r_1 \dout_j-r_1 d_{ij}+ \frac{r_1}{r_2} \din_i - \frac{r_1}{r_2} \underset{k \in \Gamma_1}{\dashsum}{\din_k}+r_2 \dout_i-r_2 d_{ij}+\\
   &\frac{r_2}{r_1} \din_j-\frac{r_2}{r_1} 
   \underset{k \in \Gamma_2}{\dashsum}{\din_k}+2 d_{ij} =\\
   &\Big(1+\frac{r_1}{r_2}\Big)\eta - \Big(1-\frac{r_1}{r_2}\Big) \din_i -\Big(1-\frac{r_2}{r_1}\Big) \din_j- \frac{r_1}{r_2} \underset{k \in \Gamma_1}{\dashsum}{\din_k}-\frac{r_2}{r_1} 
   \underset{k \in \Gamma_2}{\dashsum}{\din_k} \geq 0,
\end{align*}
where the last inequality follows from the definition of $\eta$ given by~\eqref{etadef} and the identity $r_1 + r_2 = 2$.

We now show that equalities~\eqref{c3} are implied by equalities~\eqref{c1}.
In the following we prove that all equalities of the form~\eqref{c3} with $l=1$ are implied by equalities of the form~\eqref{c1} with $l=1$; the proof of the case with $l=2$ follows from a similar line of arguments.
Using~\eqref{g1} to eliminate $\bar \lambda_{jik}$
for $i,j \in \Gamma_1$ and $k \in \Gamma_2$, and using~\eqref{g3} to eliminate $\bar \mu_i$, $i \in \Gamma_1$, it follows that equalities~\eqref{c1} with $l=1$ can be equivalently written as:
\begin{equation}\label{f1}
   \sum_{k \in \Gamma_2} {\bar\lambda_{ijk}}
 + \frac{1}{2}\sum_{k \in \Gamma_1 \setminus \{i,j\}} {(\bar\lambda_{ijk} + \bar\lambda_{jik}-\bar\lambda_{kij})}
 = \din_i + r_2 \dout_j-d_{ij}  -\eta, \quad \forall i < j \in \Gamma_1.
\end{equation}
Using~\eqref{g3}-\eqref{g5}, inequalities~\eqref{c3} can be written as:
\begin{equation}\label{f2}
    \sum_{\substack{j \in \Gamma_1 \setminus \{i\},\\ k \in \Gamma_2}}{\bar\lambda_{ijk}} +\sum_{j <k \in \Gamma_1 \setminus \{i\}}{\bar\lambda_{ijk}}= r_2 \sum_{j \in \Gamma_1}{\dout_j}-\din_i - r_2 \dout_i -(|\Gamma_1|-1) \eta , \quad \forall i \in \Gamma_1,
\end{equation}
where we used the identity $r_2 \sum_{j \in \Gamma_1}{\dout_i}=r_1 \sum_{j \in \Gamma_2}{\dout_i}$.
Moreover, it can be checked that
$$\sum_{j \neq k \in \Gamma_1 \setminus \{i\}} {(\bar\lambda_{ijk} + \bar\lambda_{jik}-\bar\lambda_{kij})} = 2 \sum_{j <k \in \Gamma_1 \setminus \{i\}}{\bar\lambda_{ijk}}.$$
Therefore, to show that equalities~\eqref{f2} are implied by equalities~\eqref{f1}, it suffices to have: 
$$
\sum_{j \in \Gamma_1 \setminus \{i\}}{(\din_i + r_2 \dout_j-d_{ij}  -\eta)} =
r_2 \sum_{j \in \Gamma_1}{\dout_j}-\din_i - r_2 \dout_i -(|\Gamma_1|-1) \eta,
$$
whose validity follows since $\sum_{j \in \Gamma_1 \setminus \{i\}}{(\din_i -d_{ij})} = (|\Gamma_1|-1) \din_i - |\Gamma_1| \din_i = \din_i$.

Therefore, it remains to prove the validity of equalities~\eqref{c1}. 
First consider the case with $l=1$. By~\eqref{g1} and nonnegativity of $\bar \lambda_{ijk}$ for all $(i,j,k)\in \Omega$, we deduce that
$$
\bar \lambda_{ijk} + \bar \lambda_{jik} \geq  {\rm abs}\left(\frac{d_{jk}-d_{ik}}{n/2}+\frac{\din_{i}-\din_{j}}{|\Gamma_2|}\right)
=\frac{1}{|\Gamma_2|}{\rm abs}\Big((\din_i+r_2 d_{jk})-(\din_j+r_2 d_{ik})\Big), \quad \forall i, j \in \Gamma_1, k \in \Gamma_2, 
$$
where ${\rm abs}(\cdot)$ denote the absolute value function. Hence, using~\eqref{g3} to eliminate $\bar \mu_i$, we conclude that equalities~\eqref{c1} with $l=1$ can be satisfied if
\begin{equation}\label{s1}
\frac{1}{2}\sum_{k \in \Gamma_1 \setminus \{i,j\}} {(\bar\lambda_{ijk} + \bar\lambda_{jik}-\bar\lambda_{kij})} \leq   \underset{k \in \Gamma_2}{\dashsum}{\min\{r_2 d_{ik} + \din_j , r_2 d_{jk} + \din_i\}} -d_{ij} - \eta, \quad \forall i < j \in \Gamma_1,
\end{equation}
where we used the identity ${\rm abs}(a-b) = a+b - 2 \min\{a, b\}$.
Letting $\bar \lambda_{ijk} = 0$ for all $(i,j,k) \in \Gamma_1$ and using~\eqref{ass1}, we conclude that inequalities~\eqref{s1} are valid. Similarly, it can be checked that inequalities~\eqref{c2} with $l=2$ can be satisfied if 
\begin{equation}\label{s2}
\frac{1}{2}\sum_{k \in \Gamma_2 \setminus \{i,j\}} {(\bar\lambda_{ijk} + \bar\lambda_{jik}-\bar\lambda_{kij})} \leq   \underset{k \in \Gamma_1}{\dashsum}{\min\{r_1 d_{ik} + \din_j , r_1 d_{jk} + \din_i\}} -d_{ij} - \frac{r_1}{r_2}\eta, \quad \forall i < j \in \Gamma_2.
\end{equation}
Letting $\bar \lambda_{ijk} = 0$ for all $(i,j,k) \in \Gamma_2$ and using~\eqref{ass2}, we conclude that inequalities~\eqref{s2} are valid and this completes the proof.
\end{proof}

We now consider the question of uniqueness of the optimal solution. To this end, we make use of the following result:

\begin{proposition}[Part~(iv) of Theorem~2 in~\cite{Man79}]\label{mang}
Consider an LP whose feasible region is defined by $Ax =b$ and $C x \leq d$, where $x \in \R^n$ denotes the
vector of optimization variables, and $b,d, A, C$ are vectors and matrices of appropriate dimensions.
Let $\bar x$ be an optimal solution of this LP and denote by $\bar u$ the dual optimal solution corresponding to the inequality constraints.
Let $C_i$ denote
the $i$-th row of $C$. Define $Q =\{i : C_i \bar x = d_i, \; \bar u_i > 0\}$, $L =\{i : C_i \bar x = d_i, \; \bar u_i = 0\}$. Let $\mathscr{C}_Q$ and $\mathscr{C}_L$
be the
matrices whose rows are $C_i$, $i \in Q$ and $C_i$, $i \in L$, respectively.
Then $\bar x$ is the unique optimal solution of the LP, if there exists no $x$ different from the zero vector satisfying
\begin{equation}\label{uc}
Ax = 0, \quad \mathscr{C}_Q x = 0, \quad \mathscr{C}_L x \leq 0.
\end{equation}
\end{proposition}

We are now ready to establish our uniqueness result:

\begin{proposition}\label{th: uniqueness}
Suppose that inequalities~\eqref{ass1} and~\eqref{ass2} are strictly satisfied. Then $\bar X$ defined by~\eqref{optsol} is the unique optimal solution of Problem~\eqref{lp:pK2}. 
\end{proposition}
\begin{proof}
Consider the dual certificate $(\bar \lambda, \bar \mu, \bar \omega, \bar \sigma)$ constructed in the proof of Proposition~\ref{Th: optimality}. We consider a slightly modified version of this certificate by redefining $\eta$ defined in~\eqref{etadef} as follows:
$$
\eta = \frac{r_2}{2}\Bigg(\Big(1-\frac{r_1}{r_2}\Big) \max_{k \in \Gamma_1}\din_k +\Big(1-\frac{r_2}{r_1}\Big) \min_{k\in \Gamma_2}\din_k+ \frac{r_1}{r_2} \underset{k \in \Gamma_1}{\dashsum}{\din_k}+\frac{r_2}{r_1} 
   \underset{k \in \Gamma_2}{\dashsum}{\din_k}\Bigg)+\epsilon,
$$
for some $\epsilon > 0$. This in turn will imply that $\bar \sigma_{ij} > 0$ for all $i \in \Gamma_1$, $j \in \Gamma_2$. Notice that this is an admissible modification as we are assuming that inequalities~\eqref{ass1} and~\eqref{ass2} are strictly satisfied. In addition, we can choose
\begin{align*}
    &\bar \lambda_{ijk} > \max\Big\{0,\frac{d_{jk}-d_{ik}}{n/2}+\frac{\din_{i}-\din_{j}}{|\Gamma_2|}\Big\}, \; \bar \lambda_{jik} > \max\Big\{0,\frac{d_{ik}-d_{jk}}{n/2}+\frac{\din_{j}-\din_{i}}{|\Gamma_2|}\Big\}, \; \forall i, j \in \Gamma_1, k \in \Gamma_2 \\
&\bar \lambda_{ijk}  > \max\Big\{0,\frac{d_{jk}-d_{ik}}{n/2}+\frac{\din_{i}-\din_{j}}{|\Gamma_1|}\Big\}, \; \bar \lambda_{jik} > \max\Big\{0,\frac{d_{ik}-d_{jk}}{n/2}+\frac{\din_{j}-\din_{i}}{|\Gamma_1|}\Big\}, \; \forall i, j \in \Gamma_2, k \in \Gamma_1. 
\end{align*}
Therefore, by Proposition~\ref{mang}, the matrix $\bar X$ defined by~\eqref{optsol} is the unique optimal solution of Problem~\eqref{lp:pK2}, if there exists no $X \neq 0$ satisfying:
\begin{align}
    &\sum_{j=1}^n{X_{ij}} = 0, \quad \forall 1 \leq i \leq n \label{cs2}\\
    & X_{ij} = 0, \quad \forall i \in \Gamma_1, j \in \Gamma_2 \label{cs3}\\
    & X_{ij}+X_{ik} = X_{ii}+X_{jk}, \quad \forall i,j \in \Gamma_1, k \in \Gamma_2, \; {\rm or} \; i,j \in \Gamma_2, k \in \Gamma_1. \label{cs4}
\end{align}
From~\eqref{cs2} and~\eqref{cs3} it follows that
\begin{equation}\label{cs5}
    X_{ii} + \sum_{j \in \Gamma_1}{X_{ij}} = 0, \quad \forall i \in \Gamma_1, \quad 
    X_{ii} + \sum_{j \in \Gamma_2}{X_{ij}} = 0, \quad \forall i \in \Gamma_2.
\end{equation}
Moreover, from equations~\eqref{cs4} we deduce that
\begin{equation}\label{cs6}
X_{ij} = \frac{X_{ii}+X_{jj}}{2}, \quad \forall i,j \in \Gamma_1 \; {\rm or} \; i,j \in \Gamma_2.
\end{equation}
Substituting~\eqref{cs6} in~\eqref{cs5} we obtain:
$
X_{ii} = X_{ij} = 0$ for all $i,j \in \Gamma_1$ and for all $i,j \in \Gamma_2$, 
which together with~\eqref{cs3} completes the proof.
\end{proof}

To have an intuitive understanding of our proximity condition, consider the special case with $\din_i = \alpha$ for all $i \in [n]$ and $|\Gamma_1| = |\Gamma_2|$. This assumption is for instance satisfied, if $|\Gamma_1| = |\Gamma_2|$ and $\Gamma_1$ and $\Gamma_2$ are uniformly distributed on two spheres $S_1$ and $S_2$ of equal radii. Denote by $c_1$ the center of $S_1$. By $\Gamma_1$ being uniformly distributed on $S_1$, we imply that for any rotation matrix $R\in \mathbb R^{d\times d}$ for which there exist two points $x^i,x^j\in \Gamma_1$ such that $R(x^i-c_1)=(x^j-c_1)$, we have $R(\Gamma_1-c_1)=(\Gamma_1-c_1)$. In this setting we have  $\eta = 1$ and assumptions~\eqref{ass1} and~\eqref{ass2} simplify to:
$$
\underset{k \in \Gamma_2}{\dashsum}{\min\{ d_{ik}, d_{jk} \}} \geq d_{ij}, \; \forall i < j \in \Gamma_1, 
\quad
\underset{k \in \Gamma_1}{\dashsum}{\min\{d_{ik},  d_{jk}\}} \geq d_{ij}, \; \forall i < j \in \Gamma_2.
$$   
These inequalities require that the distance between any two points $x_i,x_j$ in the same cluster should be smaller than the average distance between the sets $\{x_i, x_j\}$ and $\{x_k\}$, where $x_k$ denotes a point in the other cluster. Using our proximity condition~\eqref{ass1}-~\eqref{ass2}, we next obtain a recovery guarantee for the LP relaxation under a popular stochastic model for the input data.

\subsection{Recovery guarantee for the stochastic sphere model}
\label{sec:recovery}

It is widely understood that worst-case guarantees for optimization algorithms are often too pessimistic. A recent line of research in
data clustering is concerned with obtaining sufficient conditions under which a planted clustering corresponds to the unique optimal solution of a convex relaxation under suitable stochastic models~\cite{Awaetal15,MixVilWar17,IduMixPetVil17,LiLiLi20,AntoAida20}. Such conditions are often referred to as (exact) \emph{recovery} conditions and are used to compare the strength of various convex relaxations for NP-hard problems. Henceforth, we say that an optimization problem \emph{recovers} the planted clusters if its unique optimal solution corresponds to the planted clusters.

Perhaps the most popular generative model for K-means clustering is the \emph{stochastic ball model}, where the points are sampled from uniform distributions on $K$ unit balls in $\R^m$. As before, we let $K=2$ and we denote by $\Delta$ the distance between the ball centers. Notice that the question of recovery only makes sense when $\Delta > 2$.
We denote by $\Gamma_1$ and $\Gamma_2$ the set of points sampled from the first and second balls, respectively and without loss of generality we assume $|\Gamma_1| \leq |\Gamma_2|$. In the following, whenever we say \emph{with high probability}, we mean the probability tending to one as $n \rightarrow \infty$.
In~\cite{LiLiLi20}, the authors prove that the Peng-Wei SDP relaxation defined by Problem~\eqref{SDP} recovers the planted clusters with high probability if 
\begin{equation}\label{sdprec}
\Delta > 2 \Big(1+\sqrt{\frac{1}{r_1(m+2)}}\Big),
\end{equation} 
where $r_1$ is defined by~\eqref{rdef}. In the special case of equal-size clusters, \ie $|\Gamma_1|=|\Gamma_2|$, the authors of~\cite{Awaetal15} show that the Peng-Wei SDP relaxation recovers the planted clusters with high probability, if $\Delta > 2\sqrt{2}(1 + \frac{1}{\sqrt{m}})$, while the authors of~\cite{IduMixPetVil17} show that the same SDP recovers the planted clusters with high probability if $\Delta > 2(1 + \frac{2}{m})$. 
Again, for equal-size clusters, in~\cite{AntoAida20}, the authors show that Problem~\eqref{lp:pK2} recovers the planted clusters with high probability, if $\Delta > 1 + \sqrt{3}$.

In this section, we obtain a recovery guarantee for the LP relaxation for two clusters of arbitrary size. For simplicity, we consider a slightly different stochastic model, which we refer to as the \rm{stochastic sphere model} (SSM),
where points in each cluster are sampled from a sphere (\ie the boundary of a ball). We prove that our deterministic condition given by inequalities~\eqref{ass1}-\eqref{ass2} implies that
Problem~\eqref{lp:pK2} recovers the planted
clusters with high probability, if 
\begin{equation}\label{lprec}
\Delta > 1+ \sqrt{ 1+ \frac 2{r_1}}.
\end{equation}  
We should mention that, at the expense of a significantly longer proof,  one can obtain the same recovery guarantee~\eqref{lprec} for the LP relaxation under the SBM. We do not include the latter result in this paper because, while the proof is more technical and longer, it does not contain any new ideas and closely follows the path of our proof for the SSM.
Indeed, in case of equal-size clusters; \ie $r_1=1$, inequality~\eqref{lprec} simplifies to the recovery guarantee of~\cite{AntoAida20} for SBM: $\Delta > 1+\sqrt{3}$.  
As we detail in the next section, condition~\eqref{lprec} can be significantly improved via a more careful selection of the dual certificate.

Throughout this section, for an event $A$, we denote by $\prob(A)$ the probability of $A$. For a random variable $Y$,  we denote by $\avg[Y]$ its expected value. In case of a multivariate random variable $X_{ij}$, the conditional expected value in $j$, with $i$ fixed, will be denoted either with $\avg_i[X]$ or with $\avg^j[X]$. 
We denote by $\partial \B_1$ and $\partial \B_2$ the spheres corresponding to the first and second clusters, respectively.
Up to a rotation we can assume that the centers of $\partial \B_1$ and  $\partial \B_2$  are $0$ and $\Delta e_1$, respectively, where $e_1$ is the first vector of the standard basis of $\R^m$.
For a continuous function $f:\partial \B_1\to \R$ (and analogously for $\partial \B_2$), we define
$$\dashint_{\partial \B_1} f(x) d\H^{m-1}(x):=\frac 1{\H^{m-1}(\partial \B_1)}\int_{\partial \B_1} f(x) d\H^{m-1}(x),$$
where $\H^{m-1}(x)$ denotes the $(m-1)$-dimensional Hausdorff measure. 

\vspace{0.1in}

We are now ready to state our recovery result.

\begin{proposition}\label{recovery}
Suppose that the points are generated according to the SSM.
Then Problem~\eqref{lp:pK2} recovers the planted clusters with high probability if $\Delta > \Delta_0 :=1+ \sqrt{ 1+ \frac 2{r_1}}$, where $r_1$ is defined by~\eqref{rdef}.
\end{proposition}

\begin{proof}
By Proposition~\ref{th: uniqueness}, it suffices to show that for $\Delta > \Delta_0$, inequalities~\eqref{ass1}-\eqref{ass2} are strictly satisfied with high probability.
Namely, we show that there exists a universal constant $C>0$ such that, for $\Delta > \Delta_0$ we have
\begin{equation}\label{toprove}
\begin{split}
\prob \Big (\bigcap_{i,j\in \Gamma_1}\Big\{d_{ij} +\eta-\underset{k \in \Gamma_2}{\dashsum}{\min\{r_2 d_{ik} + \din_j , r_2 d_{jk} + \din_i\}} < 0\Big\}\Big ) \geq 1-Cn^2e^{-\frac{n\epsilon^2_2}{C}}
\end{split}
\end{equation}
and
\begin{equation*}
\begin{split}
\prob \Big (\bigcap_{i,j\in \Gamma_2}\Big\{d_{ij} +\frac{r_1}{r_2}\eta-\underset{k \in \Gamma_1}{\dashsum}{\min\{r_1 d_{ik} + \din_j , r_1 d_{jk} + \din_i\}} < 0\Big\}\Big ) \geq 1-Cn^2e^{-\frac{n\epsilon^2_1}{C}},
\end{split}
\end{equation*}
where $\epsilon_1,\epsilon_2 > 0$ are as defined in the statement of Lemma \ref{MasterOfProbability0}.
Since the two inequalities are symmetric, their proof is similar and we will only prove inequality~\eqref{toprove}.
To this aim, for notational simplicity, define
$$t_{ij}:=\avg^k\Big[\underset{k \in \Gamma_2}{\dashsum}{\min\{r_2d_{ik} + \avg_j[ \din_j] , r_2d_{jk} + \avg_i[\din_i]\}}\Big].$$
Then we can compute
\begin{equation}\label{toprove0}
\begin{split}
&\prob \Big (\bigcap_{i,j\in \Gamma_1}\Big\{d_{ij} +\eta-\underset{k \in \Gamma_2}{\dashsum}{\min\{r_2 d_{ik} + \din_j , r_2 d_{jk} + \din_i\}} < 0\Big\}\Big ) \\
&\geq \prob \Big(\bigcap_{i,j\in \Gamma_1}\Big\{d_{ij} +\eta-d_{ij}-2r_2+t_{ij}-\underset{k \in \Gamma_2}{\dashsum}{\min\{r_2 d_{ik} + \din_j , r_2 d_{jk} + \din_i\}} < 4\epsilon_2\Big\}\Big) \\
&\geq  \prob\Big (\Big\{\Big | \eta-2r_2 \Big|<2\epsilon_2\Big\}\cap \bigcap_{i,j\in \Gamma_1}\Big\{\Big |t_{ij}-\avg^k\Big[\underset{k \in \Gamma_2}{\dashsum}{\min\{r_2 d_{ik} + \din_j , r_2 d_{jk} + \din_i\}}\Big]\Big | <\epsilon_2\Big \} \\
&\cap\bigcap_{i,j\in \Gamma_1}\Big\{\Big |\avg^k\Big[\underset{k \in \Gamma_2}{\dashsum}{\min\{r_2 d_{ik} + \din_j , r_2 d_{jk} + \din_i\}}\Big]-\underset{k \in \Gamma_2}{\dashsum}{\min\{r_2 d_{ik} + \din_j , r_2 d_{jk} + \din_i\}}\Big | <\epsilon_2 \Big\} \Big ) \\
&\geq  1-\prob  (\{ |\eta-2r_2|\geq 2\epsilon_2\})-\prob \Big (\bigcup_{i,j\in \Gamma_1}\Big\{\Big |t_{ij}-\avg^k\Big[\underset{k \in \Gamma_2}{\dashsum}{\min\{r_2d_{ik} + \din_j , r_2d_{jk} + \din_i\}}\Big]\Big | \geq \epsilon_2\Big \}\Big )\\
&-\prob \Big (\bigcup_{i,j\in \Gamma_1}\Big\{\Big |\avg^k\Big[\underset{k \in \Gamma_2}{\dashsum}{\min\{r_2 d_{ik} + \din_j , r_2 d_{jk} + \din_i\}}\Big]-\underset{k \in \Gamma_2}{\dashsum}{\min\{r_2 d_{ik} + \din_j , r_2 d_{jk} + \din_i\}}\Big | \geq\epsilon_2 \Big\} \Big )
\end{split}
\end{equation}
The first inequality follows from \eqref{22220} in Lemma~\ref{MasterOfProbability0}, since $\Delta > \Delta_0$; the second inequality holds by set inclusion, and the
third inequality is obtained by taking the union bound. 
To complete the proof, we next estimate each of the terms in the last two lines of \eqref{toprove0}.
In the following, $C$ will always denote a universal positive constant, which may increase from one line to the next line and we will not relabel it for the sake of exposition.  First, to estimate $\prob  (\{ |\eta-2r_2|\geq 2\epsilon_2\})$, we define \begin{equation}\label{etadef1}
    \eta_1 := \frac{r_2}{2}\Bigg(\Big(1-\frac{r_1}{r_2}\Big) \max_{k \in \Gamma_1}\din_k +\Big(1-\frac{r_2}{r_1}\Big) \min_{k\in \Gamma_2}\din_k\Bigg), \quad \eta_2:= \frac{r_2}{2}\Bigg(\frac{r_1}{r_2} \underset{k \in \Gamma_1}{\dashsum}{\din_k}+\frac{r_2}{r_1} 
   \underset{k \in \Gamma_2}{\dashsum}{\din_k}\Bigg),
    \end{equation}
so that $\eta=\eta_1+\eta_2$. 
Recall that $|\Gamma_1|=r_1\frac{n}{2}$, $|\Gamma_2|=r_2\frac{n}{2}$, $r_1\in (0,1]$ and $r_2\in [1,2)$.  Since for $\Delta >  4$, the recovery follows from a simple thresholding argument, we can restrict our attention to  $\Delta \leq  4$, \ie we assume that $r_1 \geq \frac{1}{4}$. It then follows that
\begin{equation}\label{usef}
\begin{split}
\prob &  (\{ |\eta_2-\avg[\eta_2]|\geq \epsilon_2\}) \\
&= \prob\Big(\Big |r_1\Bigg(\underset{i,j \in \Gamma_1}{\dashsum}d_{ij} -  \avg\Big[\underset{i,j \in \Gamma_1}{\dashsum}d_{ij}\Big] \Bigg)+\frac{r_2^2}{r_1}\Bigg(\avg\Big[\underset{i,j \in \Gamma_2}{\dashsum} d_{ij}\Big]- \underset{i,j \in \Gamma_2}{\dashsum} d_{ij}\Bigg)\Big | \geq 2\epsilon_2 \Big) \\
&\leq \prob\Big (\Big\{\Big |\underset{i,j \in \Gamma_1}{\dashsum}d_{ij} -  \avg\Big[\underset{i,j \in \Gamma_1}{\dashsum}d_{ij}\Big]\Big|\geq \frac{\epsilon_2}{C} \Big\}  \cup \Big\{\Big |\avg\Big[\underset{i,j \in \Gamma_2}{\dashsum} d_{ij}\Big]- \underset{i,j \in \Gamma_2}{\dashsum} d_{ij}\Big | \geq \frac{\epsilon_2}{C}\Big\}\Big ) \\
&\leq  \prob\Big (\Big |\underset{i,j \in \Gamma_1}{\dashsum}d_{ij} -  \avg\Big[\underset{i,j \in \Gamma_1}{\dashsum}d_{ij}\Big] \Big|\geq \frac{\epsilon_2}{C}\Big)+ \prob\Big (\Big |\avg\Big[\underset{i,j \in \Gamma_2}{\dashsum} d_{ij}\Big]- \underset{i,j \in \Gamma_2}{\dashsum} d_{ij}\Big | \geq \frac{\epsilon_2}{C}\Big ) \leq Ce^{-\frac{n^2\epsilon_2^2}{C}}.
\end{split}
\end{equation}
The first inequality holds by set inclusion and the third inequality follows from Hoeffding's inequality (see for example Theorem 2.2.6 in~\cite{VerBookHDP}), since
$d_{ij}$, $i,j \in \Gamma_l$ are i.i.d. random variables for every $l \in \{1,2\}$ and $d_{ij} \in [0,4]$. Next we show that 
\begin{equation}\label{usef0}
\prob  \Big(\Big\{ \Big|\eta_1-\Big(2r_2-r_1-\frac{r_2^2}{r_1}\Big)\Big|\geq \epsilon_2\Big\}\Big)\leq Cne^{-\frac{n\epsilon_2^2}{C}}.
\end{equation}
For notational simplicity, we denote the $i$th point in $\Gamma_1$ by $x$. 
For any $i \in  \Gamma_1$ we have
\begin{equation}\label{0}
\begin{split}
\avg_i[\din_i]&=\dashint_{\partial \B_1}\|x-z\|^2 d\H^{m-1}(z)= \|x\|^2+\dashint_{\partial \B_1}\|z\|^2d\H^{m-1}(z)-2x^T\dashint_{\partial \B_1}zd\H^{m-1}(z) = 2.
\end{split}
\end{equation}
By symmetry, the same calculation holds for $\avg_i[\din_i]$ with $i \in \Gamma_2$.
By~\eqref{0}, we have
\begin{equation}\label{p}
\begin{split}
\avg\big[\underset{k \in \Gamma_1}{\dashsum}{\din_k}\big]&=\avg\big[\underset{k \in \Gamma_2}{\dashsum}{\din_k}\big]=2.
\end{split}
\end{equation}
Using Hoeffding's inequality together with~\eqref{0}, we get 
$$\prob  (\{ \din_k > 2+\epsilon_2 \})\leq Ce^{-\frac{n\epsilon_2^2}{C}}, \quad \prob  (\{ \din_k < 2-\epsilon_2 \})\leq Ce^{-\frac{n\epsilon_2^2}{C}}, \quad \forall k\in \Gamma_1\cup \Gamma_2.$$
Hence, by the union bound, we obtain 
$$\prob  (\{ |\max_{k \in \Gamma_1}\din_k-2|\geq \epsilon_2\})\leq Cne^{-\frac{n\epsilon_2^2}{C}}, \quad \prob  (\{ |\min_{k\in \Gamma_2}\din_k-2|\geq \epsilon_2\})\leq Cne^{-\frac{n\epsilon_2^2}{C}},$$
from which we conclude the validity of~\eqref{usef0}. 
Since by~\eqref{etadef1} and~\eqref{p} we have
$$
\avg\big[\eta_2\big] = r_2\Big(\frac{r_1}{r_2} +\frac{r_2}{r_1} 
  \Big)= \Big(r_1+\frac{r_2^2}{r_1}\Big) ,
$$
we can combine \eqref{usef} with \eqref{usef0} to conclude that
\begin{equation}\label{usef1}
\prob  (\{ |\eta-2r_2|\geq 2\epsilon_2\})\leq Cne^{-\frac{n\epsilon_2^2}{C}}.
\end{equation}
We now observe that
\begin{equation}\label{3333}
\begin{split}
\prob \Big (\bigcup_{i,j\in \Gamma_1}\Big\{\Big |t_{ij}&-\avg^k\Big[\underset{k \in \Gamma_2}{\dashsum}{\min\{r_2d_{ik} + \din_j , r_2d_{jk} + \din_i\}}\Big]\Big | \geq \epsilon_2\Big \}\Big ) \\
&\leq \prob \Big (\bigcup_{i\in \Gamma_1}\Big\{\Big | \din_i- \avg_i[ \din_i]\Big |\geq \epsilon_2/2\Big\}\Big )\leq  Cne^{-\frac{n\epsilon_2^2}{C}},
\end{split}
\end{equation}
where the first inequality follows from the linearity of expectation and the second inequality follows from the application of Hoeffding's inequality and taking the union bound. 
Finally, by Hoeffding's inequality we have
\begin{equation}\label{33331}
\begin{split}
&\prob \Big (\bigcup_{i,j\in \Gamma_1}\Big\{\Big |\avg^k\Big[\underset{k \in \Gamma_2}{\dashsum}{\min\{r_2 d_{ik} + \din_j , r_2 d_{jk} + \din_i\}}\Big]-\underset{k \in \Gamma_2}{\dashsum}{\min\{r_2 d_{ik} + \din_j , r_2 d_{jk} + \din_i\}}\Big | \geq\epsilon_2 \Big\} \Big )\\
&\leq  Cn^2e^{-\frac{n\epsilon_2^2}{C}}.
\end{split}
\end{equation}

Plugging inequalities~\eqref{usef1}-\eqref{3333}-\eqref{33331} in \eqref{toprove0}, we conclude the claimed inequality \eqref{toprove}.
\end{proof}


To prove Proposition~\ref{recovery} we made use of the following technical lemma whose proof is provided in the Appendix.

\begin{lemma}\label{MasterOfProbability0}
Suppose that the random points are generated according to the SSM.
Then the following inequalities hold provided that $\Delta > \Delta_0$:
\begin{equation}\label{22221}
4 \epsilon_1:=\inf_{i,j \in \Gamma_2}\avg^k\Big[\underset{k \in \Gamma_1}{\dashsum}{\min\{r_1d_{ik} + \avg_j[ \din_j] , r_1d_{jk} + \avg_i[\din_i]\}}\Big]- d_{ij}-2r_1 >0,
\end{equation}
\begin{equation}\label{22220}
4 \epsilon_2:=\inf_{i,j \in \Gamma_1}\avg^k\Big[\underset{k \in \Gamma_2}{\dashsum}{\min\{r_2d_{ik} + \avg_j[ \din_j] , r_2d_{jk} + \avg_i[\din_i]\}}\Big]- d_{ij}-2r_2 >0.
\end{equation}
\end{lemma}

While the recovery guarantee of Proposition~\ref{recovery} is the first of its kind for an LP relaxation of K-means clustering, it is too conservative. We demonstrate this fact via numerical simulations. We let $n=100$, $m=2$, $r_1 \in \{0.6, 0.8, 1.0\}$, and $\Delta \in [2.0:0.01:3.2]$. For each fixed configuration, we generate 20 random trials according to the SSM.
We count the number of times the optimization algorithm returns the
planted clusters as the optimal solution; dividing this number by the total number of trials, we obtain the empirical recovery rate. We use the same set up as before to solve all LPs and SDPs. Our results are shown in Figure~\ref{figure3}. Clearly, in all cases the LP outperforms the SDP. Moreover, our results indicate that the recovery threshold of the LP relaxation for SSM is significantly better than the one given by Proposition~\ref{recovery}. We should also remark that in \emph{all} these experiments, the LP relaxation is tight; \ie whenever the LP fails in recovering the planted clusters, its optimal solution is still a partition matrix. In contrast, the SDP relaxation is not tight in almost all cases for which it does not recover the planted clusters.

\begin{figure}[htbp]
 \centering
  \subfigure[$r_1 = 1$]{\label{fig3a}\epsfig{figure=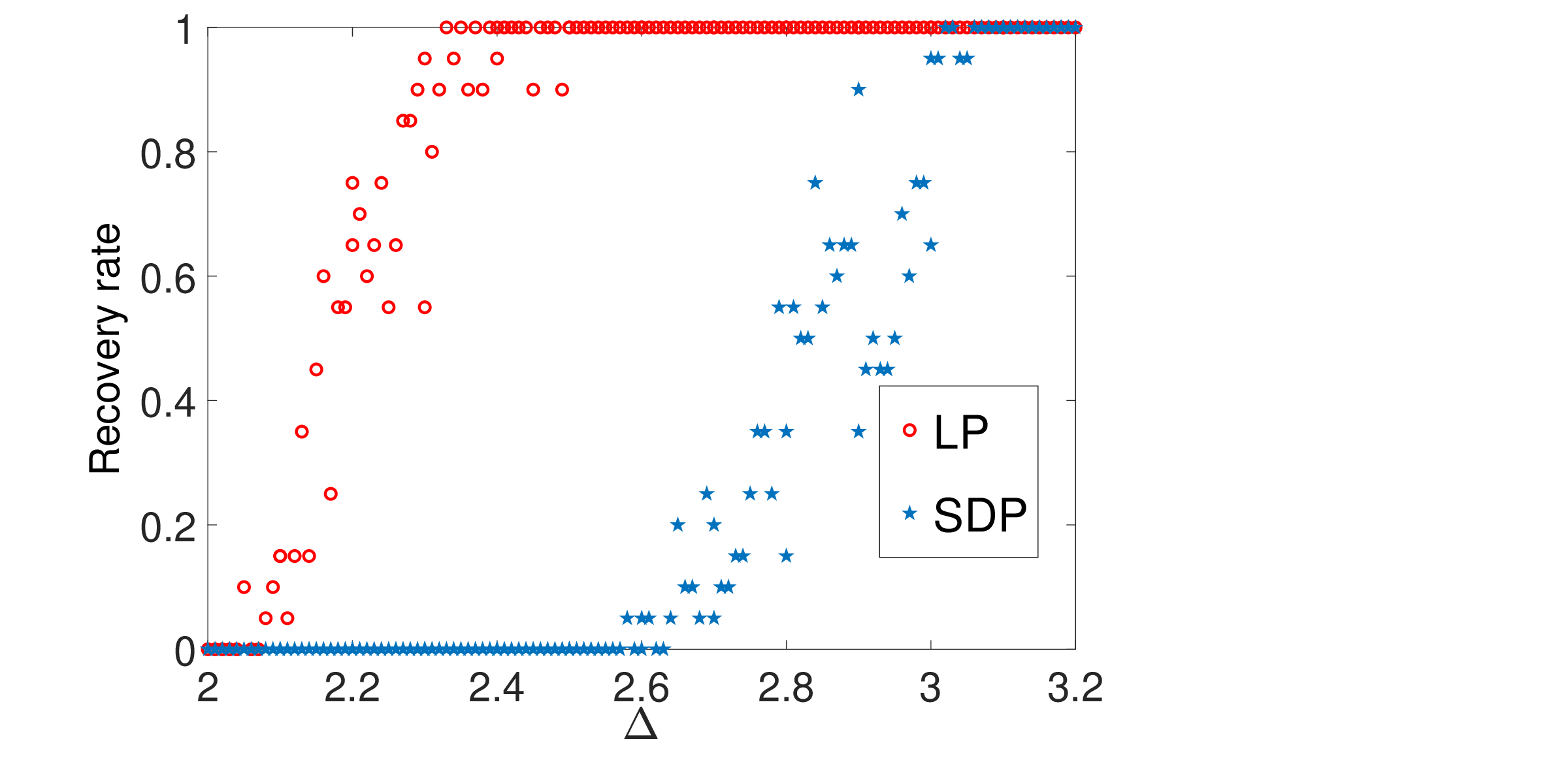, scale=0.2, trim=22mm 0mm 120mm 0mm, clip}
  }
 \subfigure [$r_1 = 0.8$]{\label{fig1b}\epsfig{figure=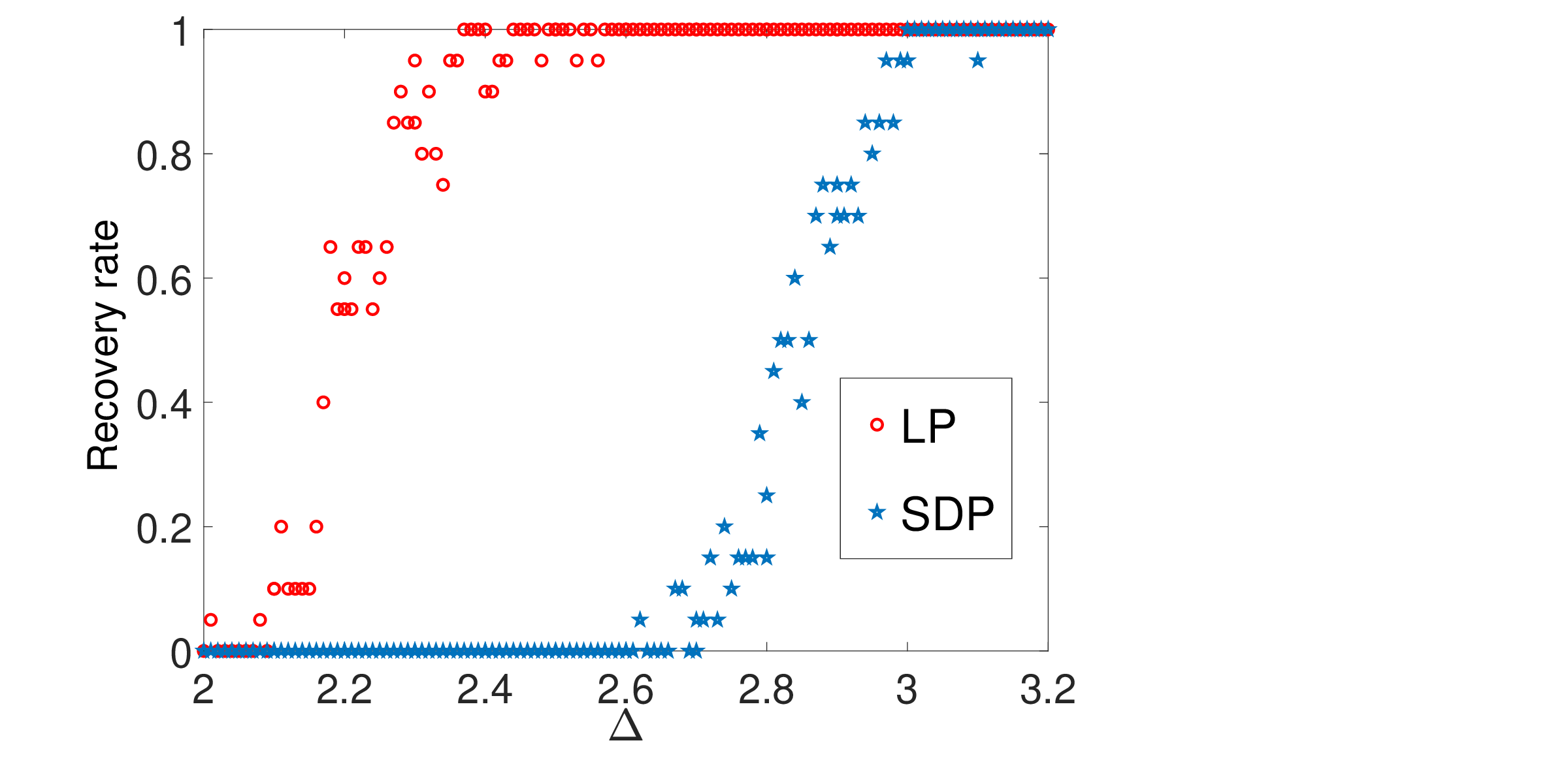, scale=0.2, trim=22mm 0mm 120mm 0mm, clip}}
 \subfigure [$r_1 = 0.6$]{\label{fig3c}\epsfig{figure=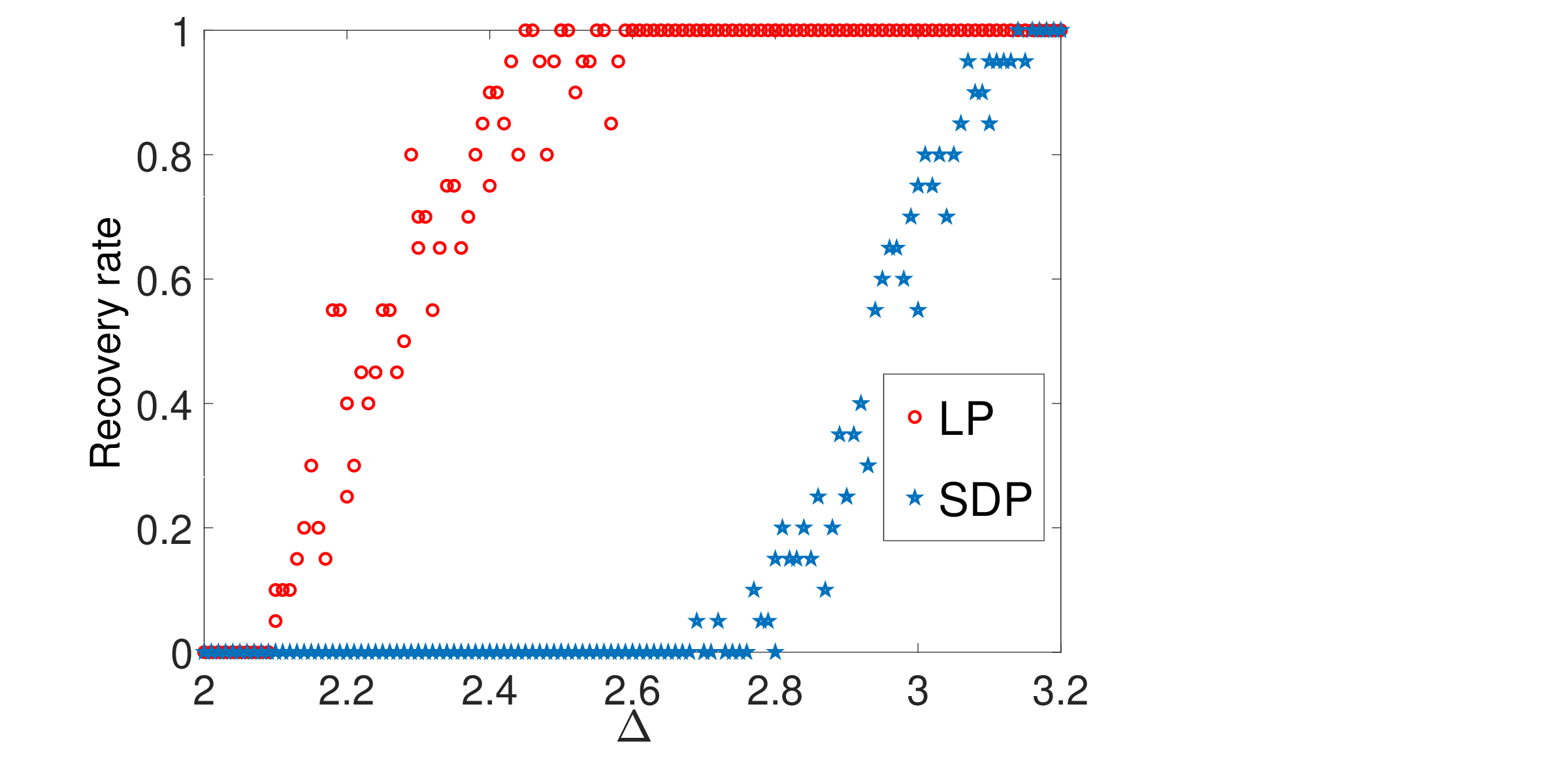, scale=0.2, trim=22mm 0mm 120mm 0mm, clip}}
  \caption{The recovery rate of the LP  versus the SDP when the input is generated according to the SSM with $n=100$ and $m=2$.
}
\label{figure3}
\end{figure}

\subsection{A stronger dual certificate}
\label{sec:certify}

By Proposition~\ref{Th: optimality}, if inequalities~\eqref{ass1} and~\eqref{ass2} are satisfied then the partition matrix $\bar X$
defined by~\eqref{optsol} is an optimal solution of Problem~\eqref{lp:pK2} and as a result $\bar X$ is an optimal solution of the K-means clustering problem. 
In the following we present a simple algorithm that leads to significantly better recovery guarantees. 
Recall that in the last step of the proof of Proposition~\ref{Th: optimality}, our task is to identify conditions under which inequalities~\eqref{s1} and~\eqref{s2} can be satisfied. In the proof we let $\bar \lambda_{ijk} = 0$ for all $(i,j,k) \in \Gamma_1$ and
$\bar \lambda_{ijk} = 0$ for all $(i,j,k) \in \Gamma_2$, which in turn gives us inequalities~\eqref{s1} and~\eqref{s2}. As we show next, a careful selection of these multipliers will lead to significantly better recovery results.
Define 
$$
   \gamma_{ij} := 2 \underset{k \in \Gamma_2}{\dashsum}{\min\{r_2 d_{ik} + \din_j , r_2 d_{jk} + \din_i\}} -2 d_{ij} - 2 \eta, \quad \forall i < j \in \Gamma_1,
$$
and 
$$
\gamma_{ij} := 2\underset{k \in \Gamma_1}{\dashsum}{\min\{r_1 d_{ik} + \din_j , r_1 d_{jk} + \din_i\}} -2 d_{ij} - 2\frac{r_1}{r_2}\eta, \quad \forall i < j \in \Gamma_2.
$$
We define $\gamma_{ji} := \gamma_{ij}$ for all $i < j$.
Then, by the proof of Proposition~\ref{Th: optimality}, the partition matrix $\bar X$ defined by~\eqref{optsol} is an optimal solution of Problem~\eqref{lp:pK2} if the following system is feasible:
\begin{align}\label{sys}
& \sum_{k \in \Gamma_{l'} \setminus \{i,j\}} {(\lambda_{ijk} + \lambda_{jik}-\lambda_{kij})} \leq \gamma_{ij}, \quad \forall i < j \in \Gamma_l, l\neq l' \in \{1,2\} \nonumber\\
& \lambda_{ijk} \geq 0, \quad \forall (i,j,k) \in \Gamma_l, l \in \{1,2\},
\end{align}
where as before we let $\lambda_{ikj} = \lambda_{ijk}$ for $j < k$.
We next present a simple algorithm whose successful termination serves as a sufficient condition for feasibility of system~\eqref{sys}:

\vspace{0.2in}

\begin{algorithm}[H]
\SetKwFunction{cert}{Certify}
\SetAlgorithmName{\cert}{\cert}{}
 \KwIn{Given $\gamma_{ij}$ for all $i < j \in \Gamma_l$, $l\in \{1,2\}$}
 \KwOut{A Boolean {\tt success}}
 Initialize $\bar \lambda_{ijk} =\bar \lambda_{jik}=\bar \lambda_{kij}=0$ for all $(i,j,k) \in \Gamma_l$, $l \in \{1,2\}$, $\bar r_{ij} = \gamma_{ij}$ for all $(i, j) \in \Gamma_l$, $l \in \{1,2\}$ and $\N = \{(i,j) \in \Gamma_l, l \in \{1,2\}: i <j, \; \bar r_{ij} < 0 \}$.\\
 \While {$\N \neq \emptyset$,}{
 set {\tt success = .false.}\\
 select some $(i,j) \in \N$ \\
  \For {each $k \in \Gamma_l \setminus \{i,j\}$,}{
  let $\omega = \min \{-\bar r_{ij}, \bar r_{ik}, \bar r_{jk}\}$. \\
  if $\omega \leq 0$, then cycle\\
 update $\bar r_{ik} \leftarrow \bar r_{ik} - \omega$, $\bar r_{jk} \leftarrow \bar r_{jk} - \omega$, and $\bar r_{ij} \leftarrow \bar r_{ij} + \omega$\\
 update $\bar \lambda_{kij} \leftarrow \bar \lambda_{kij} + \omega$, $\bar \lambda_{kji} \leftarrow \bar \lambda_{kji} + \omega$\\
  \If {$\bar r_{ij} \geq 0$,} {update success ={\tt .true.} \\
  update $\N \leftarrow \N \setminus \{(i,j)\}$ \\
  exit the k-loop}
 }
 if {\tt success = .false.}, then \Return\\
 }
\caption{The algorithm for constructing a dual certificate}
\label{alg:dual}
\end{algorithm}
\vspace{0.2in}

\begin{proposition}\label{prop:alg}
    If Algorithm~\cert terminates with {\tt  success = .true.},
    then the partition matrix $\bar X$ defined by~\eqref{optsol} is an optimal solution of Problem~\eqref{lp:pK2}. Moreover, Algorithm~\cert runs in $O(n^3)$ operations.
\end{proposition}

\begin{proof}
  We prove that the system defined by all inequalities in system~\eqref{sys} with $l=1$ is feasible; the proof for $l=2$ then follows. 
  Define 
  \begin{equation}\label{defr}
  \bar r_{ij} := \gamma_{ij}+\sum_{k \in \Gamma_2 \setminus \{i,j\}} {(\bar\lambda_{kij}-\bar\lambda_{ijk} - \bar\lambda_{jik})}, \quad \forall i < j \in \Gamma_1.
  \end{equation}  
  Then the system defined by all inequalities in system~\eqref{sys} with $l=1$ can be equivalently written as:
  \begin{equation}\label{sys2}
  \bar r_{ij} \geq 0, \quad \forall i < j \in \Gamma_1, \quad \bar \lambda_{ijk} \geq 0, \quad \forall (i,j,k) \in \Gamma_1.    
  \end{equation}
  If $\N=\emptyset$, then by letting $\bar\lambda_{ijk} =\bar\lambda_{jik}=\bar \lambda_{kij}=0$ for all $(i,j,k) \in \Gamma_1$, we obtain a feasible solution and the algorithm terminates with {\tt success = .true.}. 
  Hence, suppose that $\N \neq \emptyset$; that is, the initialization step violates inequalities $\bar r_{ij} \geq 0$ for all $(i,j) \in \N$. 
  Consider an iteration of the algorithm for some $(\bar i, \bar j) \in \N$;  we claim that if this iteration is completed with {\tt success = .true.}, all nonnegative $\bar r_{ij}$, $i < j \in \Gamma_1$ remain nonnegative (even though their values may decrease), and we will have $\bar r_{\overline {ij}} \geq 0$. Since the value of $\bar \lambda$ does not decrease over the course of the algorithm, this in turn implies that if the algorithm terminates with {\tt success = .true.},  $(\bar \lambda, \bar r)$ is feasible for system~\eqref{sys2}.  To see this, consider some $\bar k \in \Gamma_1 \setminus \{\bar i, \bar j\}$ for which we have 
  $\omega = \min \{-\bar r_{\overline{ij}}, \bar r_{\overline{ik}}, \bar r_{\overline{jk}}\} > 0$. Notice that variable $\bar \lambda_{\overline{kij}}$ (which is equal to $\bar \lambda_{\overline{kji}}$) appears only in three equations of~\eqref{defr} defining $\bar r_{\overline{ij}}$ (with positive coefficient), $\bar r_{\overline{ik}}$ (with negative coefficient), and $\bar r_{\overline{jk}}$ (with negative coefficient). Since $\bar r_{\bar{ik}} \geq \omega$, $\bar r_{\bar{jk}} \geq \omega$, and $\omega > 0$, we can increase the value of $\bar \lambda_{\overline{kij}}$ by $\omega$ and keep the system feasible, this in turn implies that the value of $\bar r_{\overline{ij}}$ will increase by $\omega$, while the values of $\bar r_{\overline{ik}}$ and $\bar r_{\overline{jk}}$ decrease by $\omega$. Therefore, if the algorithm terminates with {\tt success=.true.}, the assignment $(\bar \lambda, \bar r)$ is a feasible solution for system~\eqref{sys2}. It is simple to verify that this algorithm runs in $O(n^3)$ operation.
\end{proof}

Let us now comment on the power of Algorithm~\cert for recovering the planted clusters under the SBM. Since we do not have an explicit proximity condition, we are unable to perform a rigorous probabilistic analysis similar to that of Section~\ref{sec:recovery}. However, in dimension $m=2$, we can perform a high precision simulation as the computational cost of Algorithm~\cert is very low. First, we consider the case of equal-size clusters; \ie $r_1=r_2 =1$. We set $n=20000$ and we observe that for $\Delta > 2.14$, Algorithm~\cert terminates with {\tt success = .true.}. We then conjecture the following:

\begin{conjecture}\label{conji}
Let $m \geq 2$, and suppose that the points are generated according to the SBM with equal-size clusters. If $\Delta > 2.14$, then Problem~\eqref{lp:pK2} recovers the planted clusters with high probability.
\end{conjecture}

If true, the recovery guarantee of Conjecture~\ref{conji} is better than the recovery guarantee of the SDP relaxation given by~\eqref{sdprec} for $m \leq 202$. Clearly, any convex relaxation succeeds in recovering the underlying clusters only if the original problem succeeds in doing so. To this date, the recovery threshold for K-means clustering under the SBM for $m > 1$ remains an open question.  
Let us briefly discuss special cases $m=1$ and $m=2$. In~\cite{IduMixPetVil17}, the authors prove that a necessary condition for recovery of K-means clustering in dimension one is $\Delta > 1+\sqrt{3}$. In the same paper, the recovery threshold of the SDP relaxation in dimension one is conjectured to be $\Delta_0 = 4$ (see Section 2.3 of~\cite{IduMixPetVil17} for a detailed discussion).
In~\cite{AntoAida20}, the authors prove that Problem~\eqref{lp:pK2} recovers the planted clusters with high probability (for every $m \geq 1$), if $\Delta > 1+\sqrt{3}$. It then follows that for $m=1$, the K-means clustering problem recovers the planted clusters with high probability if and only if $\Delta > 1+\sqrt{3}$. 
For $m =2$, in~\cite{IduMixPetVil17}, via numerical simulations, the authors show that K-means clustering recovers the planted clusters with high probability only if $\Delta > 2.08$. If true, Conjecture~\ref{conji} implies that in dimension two, the recovery thresholds for K-means clustering and the LP relaxation are fairly close.  In~\cite{LiLiLi20}, the authors prove that, if
\begin{equation}\label{necCond}
\Delta < 1+\sqrt{1+\frac{2}{m+2}},
\end{equation}
then the SDP relaxation fails in recovering the planted clusters with high probability. They also state that ``it remains unclear whether this necessary condition (\ie inequality~\eqref{necCond}) is only necessary for the SDP relaxation or is necessary for the K-means itself.'' If true, Conjecture~\ref{conji} implies that inequality~\eqref{necCond} is \emph{not}
a necessary condition for the K-means clustering problem. 

We further use Algorithm~\cert to estimate the recovery threshold for the case with different cluster sizes, i.e. $r_1\neq 1$, in dimension $m=2$. 
Our results are depicted in Figure~\ref{figure2}. As can be seen from the figure, the recovery guarantee given by Algorithm~\cert is significantly better than that of Proposition~\ref{recovery}. For comparison, we have also plotted the recovery threshold for the SDP relaxation given by condition~\eqref{sdprec} for different input dimensions $m$.

\begin{figure}[htbp]
 \centering
 \epsfig{figure=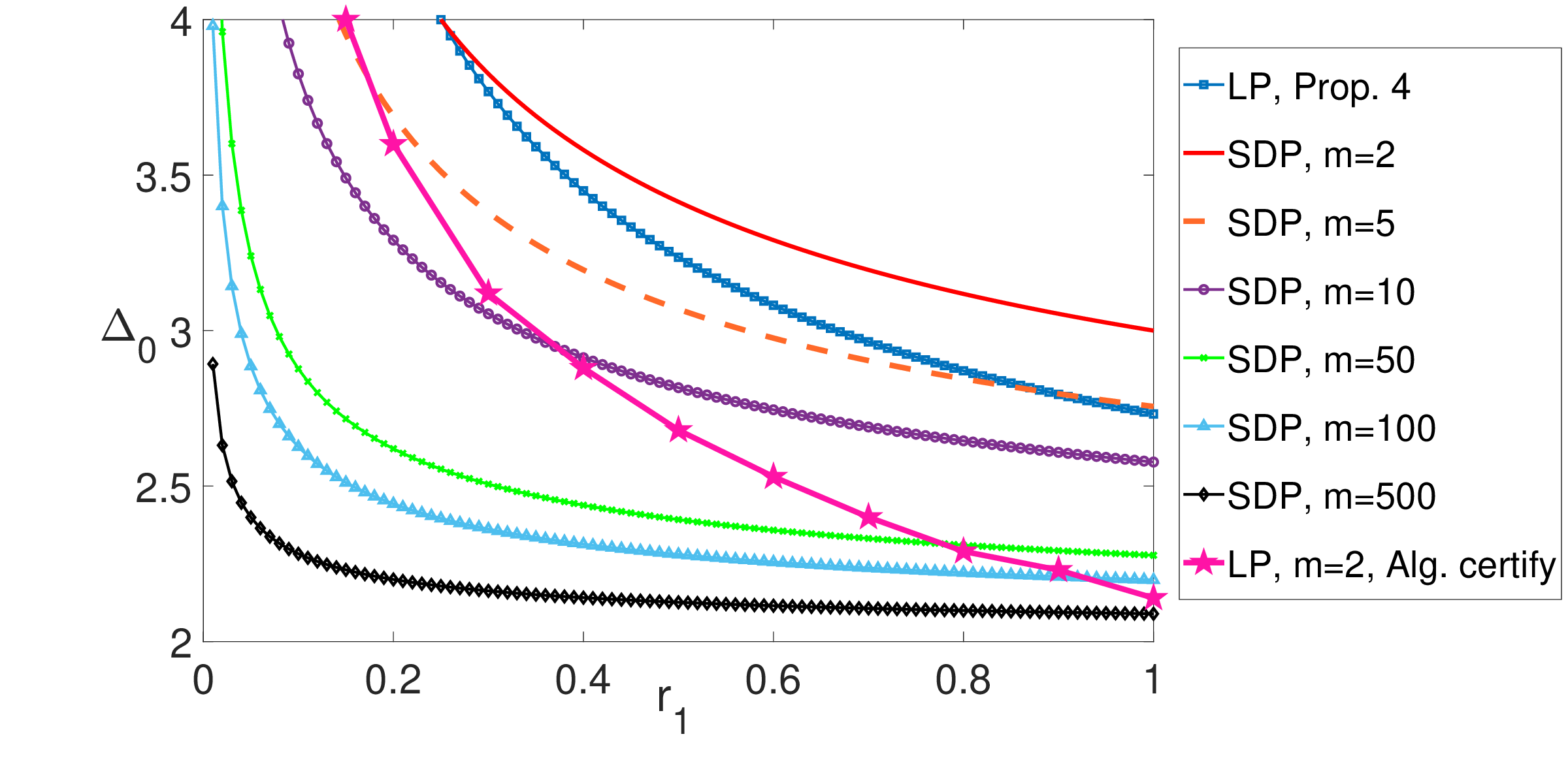, scale=0.23, trim=0mm 0mm 0mm 0mm, clip}
  \caption{Comparing the recovery threshold $\Delta_0$ of the LP relaxation versus the SDP relaxation for K-means clustering when the input is generated according to the SBM.
}
\label{figure2}
\end{figure}

\section{A Counterexample for the tightness of the LP relaxation}
\label{sec:negative}
Motivated by our previous numerical experiments, it is natural to ask whether the LP relaxation is tight with high probability under reasonable generative models. In the following we present a family of inputs for which the LP relaxation is never tight. This family of inputs subsumes as a special case the points generated by a variation of the SBM. 
To prove our result, we make use of the following lemma, which essentially states that if several sets of points have identical optimal cluster centers, then those cluster centers are optimal for the union of all points as well:

\begin{lemma}\label{samecenters}
    Let $\{x^{i,l}: i \in [n_l], l \in [L]\}$ denote a set of $N := \sum_{l\in [L]}{n_l}$ points in $\R^m$ for some $ n_1,\dots,n_L, L \in \mathbb N$, with $ n_1,\dots,n_L, L\geq 1$. For each $l \in [L]$, consider an optimal solution of Problem~\eqref{Kmeans1} for clustering the $n_l$ points 
    $\{x^{i,l}\}_{i=1}^{n_l}$ and denote by $\bar c^l\in \R^{m \times K}$ the matrix whose $k$th column denoted by $\bar c^l_k$  is equal to the center of the $k$th optimal cluster. If $\bar c^l=\bar c^{l'}$
    for all $l, l' \in [L]$, then there exists an optimal solution of Problem~\eqref{Kmeans1} for clustering the entire set of $N$ points with cluster centers $\bar c \in \R^{m \times K}$ satisfying $\bar c^l = \bar c$ for all $l\in [L]$.
\end{lemma}

\begin{proof}{Proof}
    Let $\{y^i\}_{i=1}^n$ denote a set of $n$ points in $\R^m$ that we would like to put into $K$ clusters and denote by $c \in \R^{m \times K}$ the matrix whose $k$th column denoted by $c_k$, is the center of the $k$th cluster. We start by reformulating Problem~\eqref{Kmeans1} as follows: 
    \begin{equation}\label{altform}
    \min_{c \in \R^{m \times K}} \sum_{i=1}^n{\min\Big\{\lVert y^i - c_1 \rVert_2^2, \cdots, \lVert y^i - c_K \rVert_2^2\Big\}}.
    \end{equation}
    Notice that by solving Problem~\eqref{altform} one directly obtains the optimal cluster centers and subsequently can assign each point $y^i$ to a cluster at which $\min\Big\{\lVert y^i - c_1 \rVert_2^2, \cdots, \lVert y^i - c_K \rVert_2^2\Big\}$ is attained. Now for each $l \in [L]$, define
    $$
    f_l(c) := \sum_{i=1}^{n_l}{\min\Big\{\lVert x^{i,l} - c_1 \rVert_2^2, \cdots, \lVert x^{i,l} - c_K \rVert_2^2\Big\}}. 
    $$
    It then follows that the K-means clustering problem for clustering the $n_l$
    points $\{x^{i,l}\}_{i=1}^{n_l}$ for some $l \in [L]$ can be written as:
    \begin{equation}\label{part}
      \min_{c\in \R^{m \times K}} f_l(c),    
    \end{equation}
    while the K-means clustering problem for the entire set of $N$ points can be written as
    \begin{equation}\label{all}
    \min_{c\in \R^{m \times K}} \sum_{l \in L}{f_l(c)}.
     \end{equation}
    Clearly, $\sum_{l \in L} \min_{c\in \R^{m \times K}} f_l(c) \leq \min_{c\in \R^{m \times K}} \sum_{l \in L}{f_l(c)}$. Hence, if for each $l \in [L]$, there exists a minimizer $\bar c^l$ of Problem~\eqref{part} such that $\bar c^l = \bar c^{l'}$ for all $l,l' \in [L]$, we conclude that $\bar c^l$ is a minimizer of Problem~\eqref{all} as well. 
\end{proof}

Now let us investigate the tightness of the LP relaxation defined by Problem~\eqref{lp:pK2}. First note that by Proposition~5 in~\cite{AntoAida20}, if $n \leq 4$, then the feasible region of Problem~\eqref{lp:pK2} coincides with the convex hull of the feasible region of Problem~\eqref{Kmeans2}. Hence, to find an instance for which the LP relaxation is not tight we must have $n \geq 5$.
Consider the following $n = 5$ points in dimension $m = 3$, which we refer to as the \emph{five-point input}: 
\begin{equation}\label{centers}
x^1 = \Big(0, \frac{\sqrt{3}}{3}, 0\Big), \; x^2 = \Big(\frac12, -\frac{\sqrt{3}}{6}, 0\Big), \; x^3 = \Big(-\frac12, -\frac{\sqrt{3}}{6}, 0\Big), \; x^4 =\Big(0,0,\frac12\Big), \; x^5 =\Big(0,0,-\frac12\Big).
\end{equation} 
In this case the vector of squared pair-wise distances $d=(d_{ij})_{1 \leq i < j \leq 5}$ is given by
\begin{equation}\label{distances}
d = \Big(1,1,\frac{7}{12},\frac{7}{12},1,\frac{7}{12},\frac{7}{12},\frac{7}{12},\frac{7}{12},1\Big).
\end{equation} 
By direct calculation it can be checked that the partition 
$\Gamma_1 = \{1,4\}$ and $\Gamma_2 = \{2,3,5\}$ is a minimizer of the K-means clustering problem~\eqref{Kmeans1}  
with the optimal objective value given by 
$f^*_{\rm Kmeans} = \frac{146}{72} > 2$.
Now consider the following matrix
\begin{equation}\label{lilX}
\tilde X = \frac{1}{14}
\begin{bmatrix}
6 & 1 & 1 & 3 & 3\\
1 & 6 & 1 & 3 & 3\\
1 & 1 & 6 & 3 & 3\\
3 & 3 & 3 & 5 & 0\\
3 & 3 & 3 & 0 & 5
\end{bmatrix}.
\end{equation}
%
Notice that $\tilde X$ is not a partition matrix.
It is simple to verify the feasibility of $\tilde X$ for the LP relaxation. Moreover, the objective value of
Problem~\eqref{lp:pK2} at $\tilde X$ evaluates to $\hat f_{\rm LP} = \frac{54}{28} < 2$. Denote by $f^*_{\rm LP}$ the optimal value of Problem~\eqref{lp:pK2}. Since $f^*_{\rm LP} \leq \hat f_{\rm LP} < f^*_{\rm Kmeans}$, we conclude that for the five-point input, the LP relaxation is not tight.

Now consider the following set of points, which we will refer to as the \emph{five-ball input}. Instead of $n=5$ points located at each $x^p$, $p\in \{1,\cdots,5\}$ defined by~\eqref{centers}, suppose that 
we have $n = 5 n'$ points for
some $n ' \in \mathbb N$, $n ' \geq 1$, such that for each $p \in \{1,\ldots, 5\}$, we have  $n'$ points located inside a closed ball of radius $r\geq 0$ centered at $x^p$.
In the following we show that for $r \leq 3 \times 10^{-3}$, the LP relaxation always fails in finding the optimal clusters.

\begin{proposition}\label{non-tightness}
    Consider an instance of the five-ball input for some $r \leq 3 \times 10^{-3}$. Then the optimal value of Problem~\eqref{lp:pK2} is strictly smaller than that of Problem~\eqref{Kmeans1}.
\end{proposition}

\begin{proof}
We denote by $\B_p$ the index set of points located in the $p$th ball; without loss of generality, let $\B_p = \{(p-1)n'+1,\cdots,pn'\}$ for all $p\in \{1,\cdots,5\}$. 
Let us denote these points by $y^j$, $j \in \B_p$, $p \in \{1,\cdots,5\}$.
We start by computing an upper bound on the optimal value of the LP relaxation.
Consider the matrix $\hat X$ defined as follows:
\begin{align*}
& \hat X_{ij} = \frac{6}{14 n'}, \quad \forall i, j \in \B_p, \; p \in \{1,2,3\}\\
& \hat X_{ij} = \frac{5}{14 n'}, \quad \forall i, j \in \B_p, \; p \in \{4,5\}\\
& \hat X_{ij} = \frac{1}{14 n'}, \quad \forall i \in \B_p, \; j \in \B_{p'}, \; p \neq p' \in \{1,2,3\}\\
& \hat X_{ij} = \frac{3}{14 n'}, \quad \forall i \in \B_p, \; j \in \B_{p'}, \; p \in \{1,2,3\}, \; p' \in \{4,5\}\\
& \hat X_{ij} = 0, \quad \forall i \in \B_p, \; j \in \B_{p'}, \; p \neq p' \in \{4,5\}.
\end{align*}
Clearly, $\hat X$ is not a partition matrix.
First we show that $\hat X$ is feasible for Problem~\eqref{lp:pK2}.
The equality constraint~\eqref{e1k2} is satisfied as we have ${\rm Tr}(\hat X) = \sum_{i \in [n]}{\hat X_{ii}} = 3n' (\frac{6}{14 n'})+2n' (\frac{5}{14 n'}) = 2$. Equalities~\eqref{e2k2} are also satisfied since we have $\sum_{j=1}^{n}{\hat X_{ij}}= n'(\frac{6}{14 n'})+2n'(\frac{1}{14 n'})+2n'(\frac{3}{14 n'})=1$ for all $i \in \B_p, p \in \{1,2,3\}$, and $\sum_{j=1}^{n}{\hat X_{ij}}= 3n'(\frac{3}{14 n'})+n'(\frac{5}{14 n'})+n'(0)=1$ for all $i \in \B_p, p \in \{4,5\}$.
Hence it remains to check the validity of inequalities~\eqref{e3k2}:
first note that if $i \in \B_p$, $j \in \B_{p'}$, $k \in \B_{p''}$ for $p \neq p' \neq p''$, then the validity of  $\hat X_{ij}+ \hat X_{ik} \leq \hat X_{ii} + \hat X_{jk}$ follows from the validity of inequalities~~\eqref{e3k2} at $\tilde X$ defined by~\eqref{lilX}. If $i,j,k \in \B_p$, $p \in \{1,2,3\}$, then
$\hat X_{ij}+ \hat X_{ik} = \frac{6}{14 n'}+\frac{6}{14 n'}\leq  \frac{6}{14 n'}+\frac{6}{14 n'}=\hat X_{ii} + \hat X_{jk}$  for all $i \neq j \neq k$, $j < k$. If $i,j,k \in \B_p$, $p \in \{4,5\}$, then 
$\hat X_{ij}+ \hat X_{ik} = \frac{5}{14 n'}+\frac{5}{14 n'}\leq  \frac{5}{14 n'}+\frac{5}{14 n'}=\hat X_{ii} + \hat X_{jk}$  for all $i \neq j \neq k$, $j < k$. The remaining cases, \ie when two of the points are in the same ball, while the third point is in a different ball, can be checked in a similar fashion.

Denote by $d^r_{ij}$ , $1\leq i < j \leq n$, the vector of squared pairwise distances. We have: 
\begin{align*}
&d^r_{ij} \leq 4 r^2, \quad \forall i,j \in \B_p, p \in \{1,\ldots,5\},\\
&d^r_{ij} \leq (1+2r)^2, \quad  \forall i\in \B_p, j \in \B_{p'}, \forall p \neq p' \in \{1,2,3\}, \; {\rm or} \; p \neq p' \in \{4,5\},\\ 
&d^r_{ij} \leq \Big(\sqrt{\frac{7}{12}}+2 r\Big)^2, \quad \forall i\in \B_p, j \in \B_{p'}, \forall p \in \{1,2,3\}, p' \in \{4,5\}.
\end{align*}
Hence, an upper bound $\hat g_{\rm LP}(r)$ on the optimal value of Problem~\eqref{lp:pK2} is given by
\begin{align}
g^{*}_{\rm LP}(r)&\leq 6 \binom{n'}{2} \frac{6}{14 n'} (4 r^2) +  4 \binom{n'}{2} \frac{5}{14 n'} (4 r^2)+ 6 n'^2 \frac{1}{14 n'} (1+2r)^2
 + 12 n'^2  \frac{3}{14 n'} \Big(\sqrt{\frac{7}{12}}+2 r\Big)^2 \nonumber\\
 &= n'\Big(\frac{54}{28}+20r^2+12(\frac{\sqrt{21}+1}{7})r-8\frac{r^2}{n'}\Big)\leq \Big(\frac{54}{28}+10r\Big)n':= \hat g_{\rm LP}(r),\label{ublp}
\end{align}
where the last inequality follows since by assumption $r \leq 3 \times 10^{-3}$, and
where $g^*_{\rm LP}(r)$ denotes the optimal value of Problem~\eqref{lp:pK2} for the five-ball input.

Denote by $g^*_{\rm Kmeans}(r)$ the optimal value of the K-means clustering problem~\eqref{Kmeans1} for the five-ball input.
Next, we obtain a lower bound on $g^*_{\rm Kmeans}(r)$, and show that this value is strictly larger than $\hat g_{\rm LP}(r)$ defined by~\eqref{ublp} for any $0 \leq r \leq 3 \times 10^{-3}$, implying that Problem~\eqref{lp:pK2} is not tight.
To this end,  first consider the case with $r=0$. Define the set of points $Y^l:=\{y^l, y^{n'+l}, y^{2n'+l}, y^{3n'+l}, y^{4n'+l}\}$
for all $l \in [n']$. Notice that since $r=0$, we have 
$Y^l = Y^{l'}$ for all $l, l' \in [n']$. For each $l \in [n']$, consider the K-means clustering problem with $K=2$ for five input points $Y^l$, and denote the optimal cluster centers by $(c^l_1, c^l_2)$. We then have $(c^l_1, c^l_2) = (c_1, c_2)$ for all $l \in [n']$, where $(c_1, c_2)$ denotes the optimal cluster centers for the five-point input~\eqref{centers}. Therefore, by Lemma~\ref{samecenters}, we conclude that an optimal cluster centers for the five-ball input with $r=0$, coincides with an optimal cluster centers for the five-point input. This in turn implies that 
\begin{equation}\label{add1}
g^*_{\rm Kmeans}(0) = n' f^*_{\rm Kmeans}  = \frac{146}{72} n'.
\end{equation}
%
Now let us consider the five-ball input for some $r > 0$.
We observe that 
$\sqrt{d^r_{ij}} \geq \sqrt{d^0_{ij}}-2 r$,  for all  $1 \leq i < j \leq n$.
 Since $d^0_{ij}\leq 1$ for all $1 \leq i < j \leq n$, we have
$
d^r_{ij} \geq d^0_{ij}-4 r$  for all $1 \leq i < j \leq n$.
Therefore, for every partition matrix $X$ we compute
$$\sum_{i,j \in [n]} {d^r_{ij} X_{ij}}\geq \sum_{i,j \in [n]} {d^0_{ij} X_{ij}} -4 r\sum_{i,j \in [n]} {X_{ij}}=\sum_{i,j \in [n]} {d^0_{ij} X_{ij}} -20 rn',$$
where the last equality follows from constraints~\eqref{e2k2}.
Hence, taking the minimum over all partition matrices on both sides and using~\eqref{add1}, we deduce that 
\begin{equation}\label{lowerbound}
g^*_{\rm Kmeans}(r) \geq g^*_{\rm Kmeans}(0)-20 rn'=\Big(\frac{146}{72}-20r\Big)n'.
\end{equation} 
From~\eqref{ublp} and~\eqref{lowerbound} it follows that by choosing 
$$r< \frac{\frac{146}{72}-\frac{54}{28}}{30},$$ 
we get $g^{*}_{\rm LP}(r) < g^*_{\rm Kmeans}(r)$ and this completes the proof.
\end{proof}

As a direct consequence, we find that the LP relaxation is not tight for a variant of the SBM:

\begin{corollary}\label{rem-tight}
Let $r \leq 3 \times 10^{-3}$. Suppose that the points are generated in $\R^m$, for some $m\geq 3$ according to any generative model consisting of five measures $\mu_p$, $p=1,\dots,5$, each supported in the ball $B((x^p,0,\dots,0),r)$, where $x^p$ is defined by~\eqref{centers}.
Then Problem~\eqref{lp:pK2} is not tight. 
\end{corollary}

Recall that in all our previous numerical experiments with synthetic and real-world data sets, the LP relaxation outperforms the SDP relaxation. That is, the optimal value of the LP is always at least as large as that of the SDP. Hence, one wonders whether such a property can be proved in a general setting. Interestingly, the stochastic model defined in Corollary~\ref{rem-tight} provides the first counterexample, which we illustrate via a numerical experiment. We consider the stochastic model defined in Corollary~\ref{rem-tight}, where we assume the points supported by each ball are sampled from a uniform distribution. We set $m=3$ and generate $n'=20$ points in each of the five balls to get a total of $n=100$ points. Moreover we set ball radii $r \in [0.0: 0.01: 0.5]$ and for each fixed $r$ we generate 20 random instances. Our results are depicted in Figure~\ref{figure4}, where as before we compare the LP relaxation with the SDP relaxation, with respect to their tightness rate and average relative gap. Recall that the relative gap is defined as
$g_{\rm rel} = \frac{f_{\rm LP}-f_{\rm SDP}}{f_{\rm LP}} \times 100$, where $f_{\rm LP}$ and $f_{\rm SDP}$ denote the optimal values of the LP and the SDP, respectively. That is, a negative relative gap means that the SDP relaxation is stronger than the LP relaxation. As can be seen from the figure, while the SDP is never tight, for $r \lesssim 0.15$, we often have $f_{\rm SDP} > f_{\rm LP}$.

\begin{figure}[htbp]
 \centering
  \subfigure{\epsfig{figure=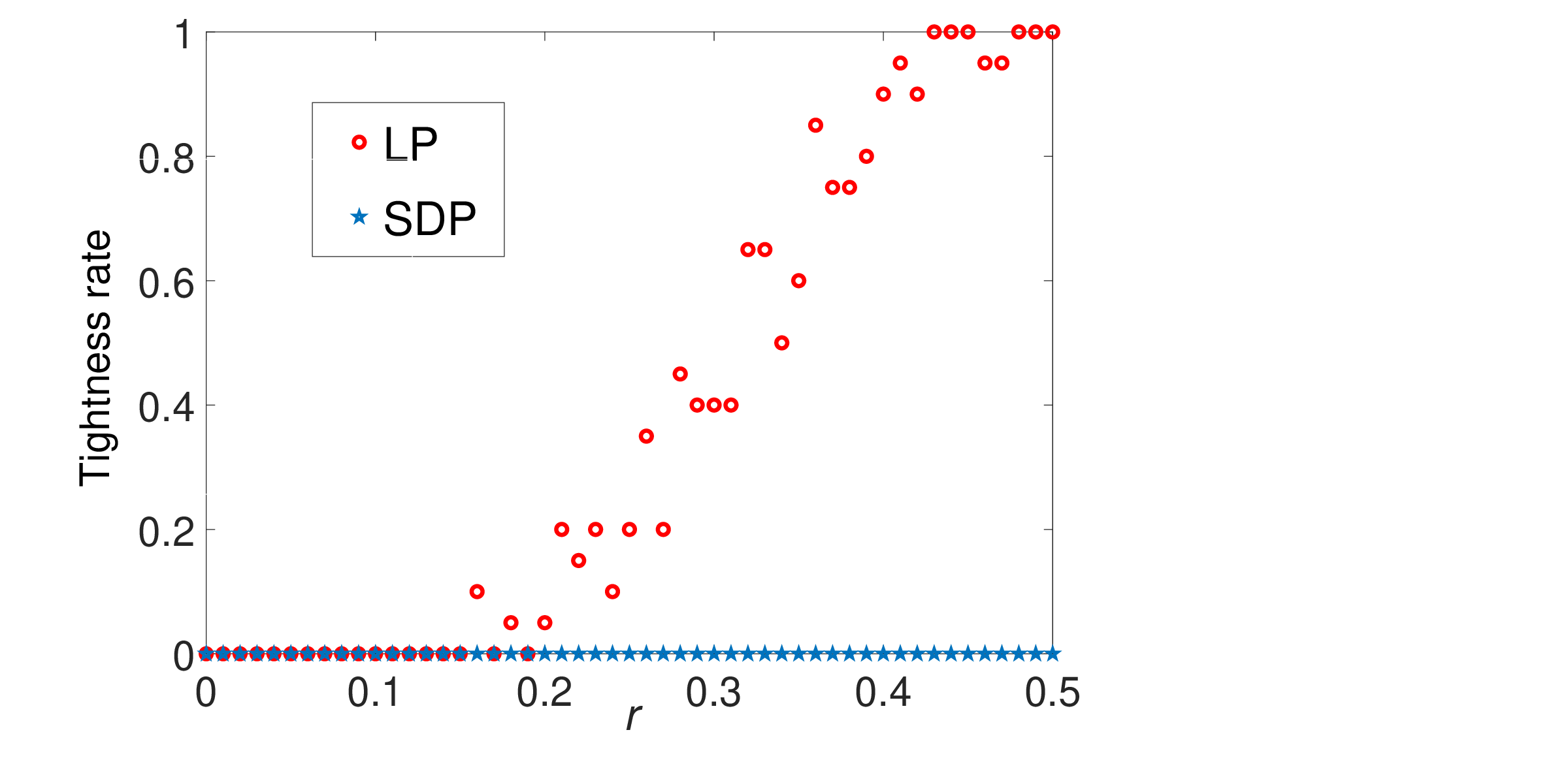, scale=0.21, trim=20mm 0mm 120mm 0mm, clip}
  }
 \subfigure{\epsfig{figure=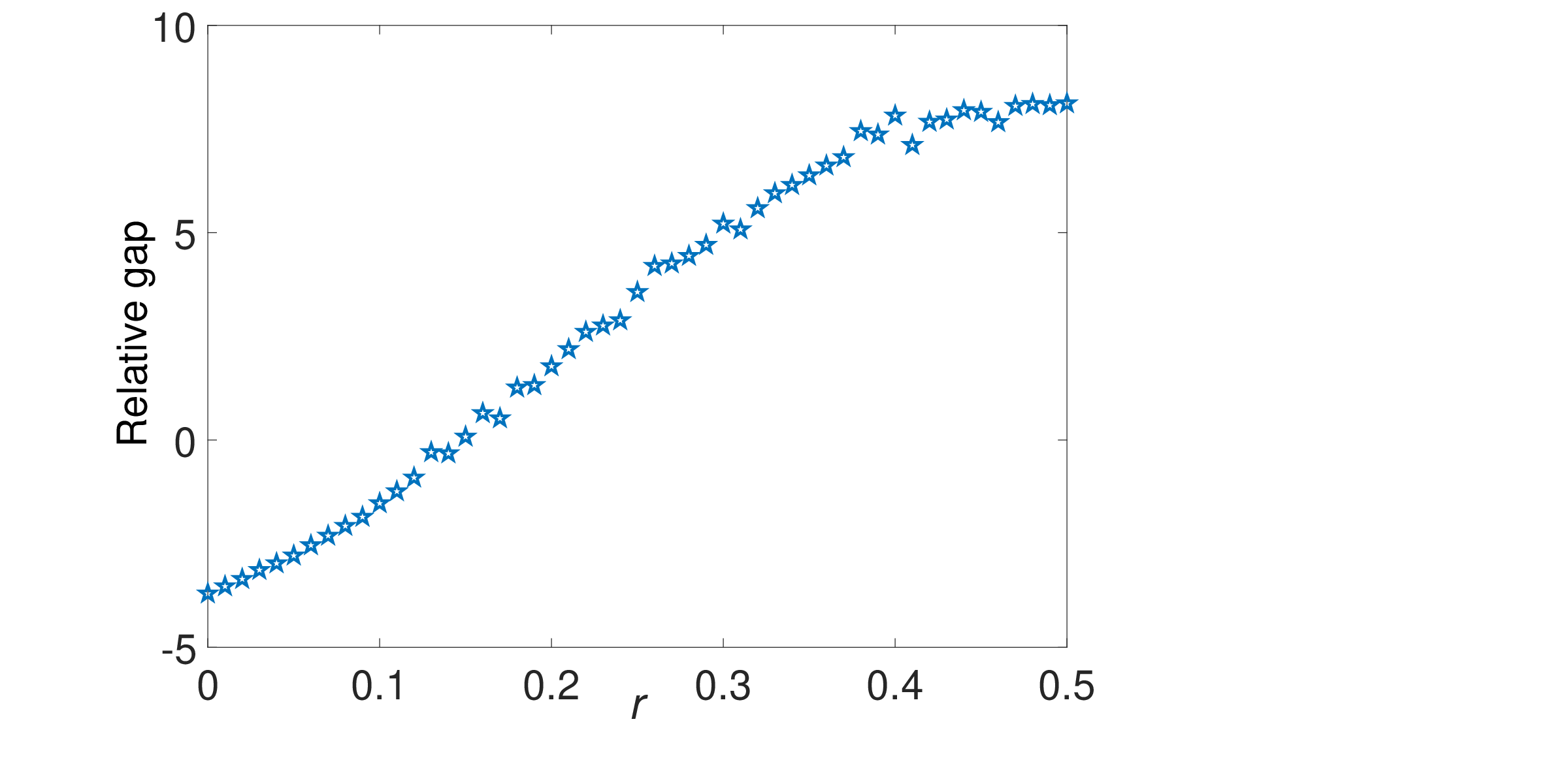, scale=0.21, trim=20mm 0mm 120mm 0mm, clip}}
  \caption{Comparing the strength of the LP versus the SDP for K-means clustering when the input is generated according to the stochastic five-ball input with $n=100$ and $m=3$.
}
\label{figure4}
\end{figure}

We conclude this section by noting that the family of examples considered above are very special and our numerical experiments with real-world data sets suggest that such special configurations do not appear in practice. Hence, it is highly plausible that one can establish the tightness of the LP relaxation under a fairly general family of inputs. We leave this as an open question.

\section{A customized algorithm for solving the LP relaxation}
\label{sec:numerics}
As we detailed in the previous section, Problem~\eqref{lp:pK} contains $\Theta(n^{t+1})$ inequalities of the form~\eqref{eq3k}. This in turn implies that even in the cheapest case; \ie $t=2$, the computational cost of solving the LP relaxation quickly increases making it impractical for many real-world applications. To further illustrate this fact, we consider three state-of-the-art LP solvers: (1)
{\tt Gurobi 11.02}~\cite{gurobi}, we use the interior-point algorithm and turn off the {\tt crossover} operator to expedite the solver, (2) {\tt PDLP 9.10}~\cite{applegate2021practical}, a primal-dual 
first-order method for solving large-scale LPs that is implemented in Google's OR-Tools,  
(3) {\tt cuPDLP-C}~\cite{lu2023cupdlpC}, an improved implementation of  {\tt cuPDLP} in the {\tt C} language, where
{\tt cuPDLP}~\cite{lu2023cupdlp} is a GPU implementation of {\tt PDLP} and is written in {\tt Julia}.

We first show that even for relatively small data sets \ie $n < 500$, none of these three solvers are able to solve Problem~\eqref{lp:pK} within $4$ hours. Again, we consider a number of data sets from the UCI machine learning repository~\cite{dua2017uci}. The data sets and their characteristics, \ie the number of points $n$ and the number of clusters $K$ are listed in columns 1-3 of Table~\ref{table2}. We solve Problem~\eqref{lp:pK} with $t=2$ using the three LP solvers. 
The size of each LP, \ie the number of variables $n_v$, the number of constraints $n_e$, and the number of nonzeros in the constraint matrix $n_z$  are listed in columns 4-6 of Table~\ref{table2}. Henceforth, all experiments are conducted on {\tt Google Colab} using an Intel(R) Xeon(R) CPU @ 2.20GHz with 8 cores and 50 GB of RAM. The GPU used is an NVIDIA L4 with 22.5 GB of RAM. 
The total time in seconds of solving the LP using each solver is listed in the last three columns of Table~\ref{table2}. As can be seen from this table, for all problems that fit in the GPU memory, {\tt cuPDLP-C} significantly outperforms other solvers. Nonetheless, for problems with $n \geq 400$, the LP becomes excessively expensive for all solvers. It is interesting to note that with the exception of {\tt glass} and {\tt WholeSale} (with $K=6$) data sets, for all instances, the LP relaxation is tight; that is, it returns a partition matrix upon termination. As we detail shortly, for {\tt glass} and {\tt WholeSale} data sets, Problem~\eqref{lp:pK} is tight if $t=3$  and $t=6$, respectively.

\begin{table}[htb]
\caption{Execution time (secs) for solving Problem~\eqref{lp:pK} with $t=2$ using different LP solvers for small data sets from UCI.}
\label{table2}
\small
\centering
\begin{threeparttable}
\begin{tabular}{|ccc|ccc|ccc|}
\hline
Data set & $n$ & $K$ & $n_v$ & $n_e$ & $n_z$ & {\tt cuPDLP-C} & {\tt PDLP} & {\tt Gurobi} \\ \hline
ruspini & $75$ & $4$ & $2850$ & $2.03\times 10^{5}$ & $8.16\times 10^{5}$ & $1.0$ & $4.0$ & $3.9$ \\
iris & $150$ & 2 & $11325$ & $1.65\times 10^{6}$ & $6.64\times 10^6$ & $5.6$ & $31.0$ & $33.4$ \\
 &  & $3$ &  &  &  & $7.9$ & $46.6$ & $45.1$ \\
 &  & $4$ &  &  &  & $7.2$ & $41.1$ & $59.8$ \\
wine & $178$ & $2$ & $15931$ & $2.77\times 10^6$ & $1.11\times 10^7$ & $14.4$ & $73.7$ & $79.4$ \\
 &  & $7$ &  &  &  & $20.0$ & $98.9$ & $195.8$ \\
gr202 & $202$ & $6$ & $20503$ & $4.06\times 10^6$ & $1.63\times 10^7$ & $131.9$ & $403.1$ & $868.7$ \\
seeds & $210$ & $2$ & $22155$ & $4.56\times 10^6$ & $1.83\times 10^7$ & $24.3$ & $207.2$ & $163.1$ \\
 &  & $3$ &  &  &  & $19.3$ & $115.6$ & $200.6$ \\
glass & $214$ & $3$\tnote{1} & $23005$ & $4.83\times 10^6$ & $1.94\times 10^7$ & $116.2$ & $845.1$ & $651.0$ \\
 &  & $6$\tnote{1} &  &  &  & $168.3$ & $864.3$ & $1394.6$ \\
CatsDogs & $328$ & $2$ & $53956$ & $1.75\times 10^7$ & $7.00\times 10^7$ & $1245.8$ & $2894.2$ & $1104.2$ \\
accent & $329$ & $2$ & $54285$ & $1.76\times 10^7$ & $7.07\times 10^7$ & $148.9$ & $713.8$ & $1395.9$ \\
 &  & $6$ &  &  &  & $271.0$ & $982.7$ & $8401.9$ \\
ecoli & $336$ & $3$ & $56616$ & $1.88\times 10^{7}$ & $7.53\times 10^{7}$ & $129.7$ & $547.3$ & $5729.5$ \\
RealEstate & $414$ & $3$ & $85905$ & $3.52\times 10^7$ & $1.41\times 10^8$ & N/A\tnote{2} & $3974.6$ & $>14400$ \\
 &  & $5$ &  &  &  & N/A & $5708.4$ &  $>14400$ \\
wholesale & $440$ & $5$ & $97020$ & $4.23\times 10^7$ & $1.69\times 10^8$ & N/A & $2944.4$ & N/A \\
 &  & $6$\tnote{1} &  &  &  & N/A & $>14400$ & N/A \\ \hline
\end{tabular}
\begin{tablenotes}
\item[1] Problem~\eqref{lp:pK} with $t=2$ is not tight. 
\item[2] Exceeds the maximum memory of the machine.
\end{tablenotes}
\end{threeparttable}
\end{table}

We next present a cutting-plane based algorithm that enables us to solve Problem~\eqref{lp:pK} for significantly larger data sets. 
In the following, we describe the main ingredients of our algorithm\footnote{The source code as well as the data sets are available at {\tt https://github.com/Yakun1125/cutLPK/tree/main}}.

\subsection{Initialization}
Llyod's algorithm and its enhanced implementations such as {\tt k-means++} are extremely fast and are often successful in finding high quality solutions for the K-means clustering problem. In our framework, we make use of {\tt k-means++}'s solution at the initialization step. First, the cost associated with this solution serves as an (often very good) upper bound on the optimal value of Problem~\eqref{lp:pK}. Second, to construct the first LP, we select inequalities~\eqref{eq3k} with $t=2$ that are satisfied tightly at the {\tt k-means++}'s solution. 
For large scale data sets, however, the number of inequalities~\eqref{eq3k} with $t=2$ that are satisfied tightly at the {\tt k-means++}'s solution might be too large; in those cases, we randomly select a subset of at most $p_{\rm init}$ of such inequalities. 
 
\subsection{Safe lower bounds}
To further expedite the cutting plane algorithm for large scale problems, in the first few iterations, we solve LPs with a larger feasibility tolerance and terminate early, if needed. In such cases the solution returned by the LP solver is often infeasible. To generate valid lower bounds on the optimal value of Problem~\eqref{lp:pK}, we make use of a technique that has been used in mixed-integer programming solvers~\cite{neumaier2004safe}. Consider the following pair of primal-dual LPs:
\begin{align*}
    \min \quad & c^T x \qquad  &\max\quad  & b^T y\nonumber\\
    \st \quad  &Ax = b, \; Q x \leq 0  \qquad
     &\st \quad & A^T y + Q^T z \leq c \nonumber \\
     & x \geq 0,  & & z\leq 0.\nonumber
\end{align*}
Denote by $(\hat x, \hat y, \hat z)$ an infeasible primal-dual pair returned by the LP solver satisfying $\hat z \leq 0$.
Define the dual residual $r$ as
$r := \max\{A^T \hat{y} + Q^T \hat{z} - c, 0\}$. We then have $c \geq A^T \hat{y} + Q^T \hat{z} - r$. Let $x^*$ denote an optimal solution for the primal LP. Using $A x^* =b$, we get $c^T x^* \geq \hat{y}^T b + \hat{z}^T Q x^* - r^T x^*$.
From $Q x^*\leq 0$ and $\hat{z}\leq 0$, we deduce that  $\hat{z}^T Q x^* \geq 0$ implying that $c^T x^* \geq \hat{y}^T b - r^T x^*$. Denote by $\bar x$ an upper bound on $x^*$. We then obtain the following ``safe lower bound'' on the optimal value of the primal LP: 
$$c^T x^* \geq \hat{y}^T b - r^T \bar{x}.$$

\subsection{Rounding}
Denote by $X_{lb}$ an optimal solution of the LP relaxation at an iteration of the cutting plane algorithm and suppose that $X_{lb}$ is not a partition matrix. It is then desirable to ``round'' this solution to obtain a partition matrix $X_{ub}$ whose cost may serve as a good upper bound on the optimal clustering cost. We make use of the rounding scheme proposed in~\cite{PenWei07}.


\subsection{The separation algorithm}
Recall that the key to the efficiency of a cutting plane algorithm to solve Problem~\eqref{lp:pK} is the efficient separation of inequalities~\eqref{eq3k}. That is, each intermediate LP should contain a small number of inequalities~\eqref{eq3k} that are as sparse as possible. Fix $i \in [n]$; we argue that if $t$ is part of the input, then the separation problem over inequalities~\eqref{eq3k} for fixed $i$ is NP-hard. To see this let us rewrite  inequalities~\eqref{eq3k} as follows:
$$\sum_{j \in S} {X_{ij}} - \sum_{j, k \in S: j < k} {X_{jk}}\leq X_{ii}, \quad S \subseteq [n] \setminus \{i\}: 2\leq |S|\leq t. 
$$
For a fixed $i$, the above inequalities are in the form of clique inequalities introduced by Padberg~\cite{Pad89}. It is well-known that separating over these inequalities is strongly NP-hard, being equivalent to maximum edge-weight clique problem. In the following we fix $t \in \{2,\ldots, K\}$ and propose an efficient heuristic separation scheme for inequalities~\eqref{eq3k}. Let $\tilde X$ be the solution to the current LP. 
For each $i \in [n]$, the separation problem is 
 to find a nonempty subset $S\subseteq[n] \setminus \{i\}$ with $2 \leq |S| \leq t$ such that
$$
w_i(S) := \sum_{j \in S} \tilde X_{i j} -\sum_{j, k \in S: j<k} \tilde X_{j k}>\tilde X_{i i},
$$
or to prove that no such subset exists. Repeated calculations can be avoided by using the relation:
$
w_i(S\cup\{k\}) = w_i(S) + \tilde X_{ik}-\sum_{v\in S}\tilde X_{vk}
$,
where we use the identity $\tilde X_{ij} = \tilde X_{ji}$ for all $i < j$.
However, this enumeration scheme is too expensive for large scale problems. To this end, we utilize the greedy strategy proposed in~\cite{marzi2019computational} to separate clique inequalities.  


\begin{algorithm}[htb]
\SetKwFunction{sep}{separate}
\SetAlgorithmName{\sep}{\sep}{}
\caption{The heuristic algorithm for separating inequalities~\eqref{eq3k}}
\KwIn{LP solution $\tilde X$, violation tolerance $\epsilon_{\rm vio}$, and $t_{\max}$}
\KwOut{A number of inequalities of the form~\eqref{eq3k} violated at $\tilde X$.}
            \For{$i \in [n]$ \textbf{in parallel}}{
                \For{$j \in [n] \setminus \{i\}$}{          
                    Initialize $S = \{j\}$, $w_i(S) = \tilde X_{ij}$, and $c = j$  \\             
                    \While{$|S| < t_{\max}$}
                        {select $k \in \{c + 1, \dots, n\} \setminus (S \cup \{i\})$ that maximizes $\gamma_i(k) = \tilde X_{ik}-\sum_{l\in S,l<k} \tilde X_{lk}$.\\
                        
                         Update $\bar{S} = S\cup\{k\} $, and $c = k$ \\
                        Compute $w_i(\bar{S}) = w_i(S) +\gamma_i(k)$. \\                      
                        \If{$ w_i(\bar{S}) > \tilde X_{ii} + \epsilon_{\rm vio} $}{
                        Add $ \sum_{k \in \bar S} {X_{ik}} \leq X_{ii} + \sum_{l, k \in \bar S: l < k} {X_{lk}}$
                        to the system of violated inequalities.
                        }
                        Update $S = \bar{S}$
                    }
                }
            }
\end{algorithm}

\vspace{0.1in}

An outline of the separation scheme is provided in Algorithm~\sep.
Clearly, this algorithm is not exact, as greedy expansion of $S$ may lead to the exclusion of some violated inequalities. 
However, in our numerical experiments, for all the problems, Algorithm~\sep
was enough to close the optimality gap $r_g$. To separate inequalities~\eqref{eq3k}, we run $n$ separation problems, each for a fixed $i \in [n]$, in parallel. We make use of {\tt OpenMP}~\cite{openmp} to do the parallel computations.

\begin{algorithm}[htb]
\SetKwFunction{cut}{Iterative LP solver}
\SetAlgorithmName{\cut}{\cut}{}
\KwIn{Data points $\{x^i\}_{i=1}^n$, number of clusters $K$, optimality tolerance $\epsilon_{\rm opt}$, initial number of inequalities $p_{\rm init}$, number of inequalities added at each round of cut generation $p_{\max}$, and the solver time limit for each intermediate LP $T$.}
\KwOut{Partition matrix $X_{ub}$ and optimality gap $r_g$.}
\textbf{Initialize:} 
    Set the lower bound $f_{lb} = -\infty$ and the optimality gap $r_g = +\infty$.
    Run {\tt k-means++} to get a partition matrix $X_{ub}$.
    Set the upper bound $f_{ub}$ as the cost of $X_{ub}$.    
    Randomly select at most $p_{\rm init}$ of inequalities~\eqref{eq3k} with $t = 2$ that are active at $X_{ub}$ in the LP. Set $t_{\max} = 2$.

\While{there exists a violated inequality of form~\eqref{eq3k},}{
    Solve the LP to obtain a safe lower bound $\bar{f}_{lb}$ and an optimal solution $X_{lb}$. \\
    \If{$\bar{f}_{lb}>f_{lb}$,}
       {Update $f_{lb} = \bar{f}_{lb}$}
    Round $X_{lb}$ and get a partition matrix $\bar{X}_{ub}$ with cost $\bar{f}_{ub}$.\\
    \If{$\bar{f}_{ub} < f_{ub}$,}
        {Update $X_{ub} = \bar{X}_{ub}$ and
        $f_{ub} = \bar{f}_{ub}$}
    Update $r_{g} = (f_{ub}-f_{lb})/f_{ub}$\\
    \If{$r_{g} \leq \epsilon_{\rm opt}$,}{Terminate}
    Remove from the LP, inequalities~\eqref{eq3k} that are not satisfied tightly at $X_{lb}$.\\
    Run~\sep to obtain 
    a number of inequalities of the form~\eqref{eq3k}
    with $t \leq t_{\max}$ that are violated at $X_{lb}$. 
    Add at most $p_{\max}$ of the most violated inequalities to the LP.\\
    
    \If{the number of violated inequalities returned by Algorithm~\sep is small,}
    {Update $t_{max} = \min\{K, t_{max} + 1\}$}
}
\label{alg:Cutting}
\caption{The cutting plane algorithm for solving Problem~\eqref{lp:pK}}
\end{algorithm}

An overview of iterative cutting-plane based algorithm is given in Algorithm~\cut.  Notice that in~\cut, the sparsity of the inequalities added to the LP is controlled by parameter $t_{\max}$, which is initially set to $t_{\max} = 2$, and is increased by one, only if the number of violated inequalities~\eqref{eq3k} with $t \leq t_{\max}$ that are found by Algorithm~\sep is below a certain threshold.

\subsection{Numerical Experiments}

We start by testing~\cut using our three LP solvers as before for small data sets from UCI. Results are shown in Table~\ref{tab:smalldata}. Notice that~\cut is not a deterministic algorithm; that is, at initialization,
if the total number of eligible inequalities exceeds $p_{\rm init}$, a random subset is selected and added to the LP. Hence, for a fair comparison, we run each instance $5$ times and report average time and average number of iterations along with standard deviations.
We set the optimality tolerance $\epsilon_{\rm opt} = 10^{-4}$. Results are summarized in Table~\ref{tab:smalldata}.
As can be seen from the table, even with the cut generation scheme, {\tt Gurobi} is unable to solve instances with more than $400$ data points within $4$ hours of CPU time.
On the other hand, both first-order method LP solvers solve all instances within the time limit. \emph{Surprisingly, in all cases, the LP relaxation is tight; \ie the optimality gap $r_g$ upon termination is smaller than the optimality tolerance $\epsilon_{\rm opt}$.} Moreover, {\tt cuPDLP-C} is on average $5.7$ times faster than {\tt PDLP}. As a byproduct, we can verify that out of $19$ instances, the SDP relaxation is tight for $2$ instances only. In addition, {\tt k-means++}'s solution is optimal for $12$ instances. 

\begin{table}[htb]
\caption{Average run time and average number of iterations of the cutting plane algorithm with different LP solvers for small data sets. The numbers in parentheses denote the standard deviations across five random trials. No parentheses implies identical results for all trials.}
\label{tab:smalldata}
\small
\centering
\begin{threeparttable}
\begin{tabular}{|ccc|cc|cc|cc|}
\hline
\multirow{2}{*}{Data set} & \multirow{2}{*}{$n$} & \multirow{2}{*}{$K$} & \multicolumn{2}{c|}{\tt cuPDLP-C} & \multicolumn{2}{c|}{\tt PDLP} & \multicolumn{2}{c|}{\tt Gurobi} \\ \cline{4-9} 
 &  &  & time(s) & iteration & time(s) & iteration & time(s) & iteration \\ \hline
ruspini & $75$ & $4$\tnote{*} & $0.5$ & $1.0$ & $3.0$ & $1.0$ & $1.1$ & $1.0$ \\
iris & $150$ & $2$ & $5.1$  & $1.0$  & $18.8$  & $1.0$  & $20.7$  & $1.0$  \\
 &  & $3$ & $4.4$  & $1.0$  & $23.3$  & $1.0$  & $23.8$  & $1.0$  \\
 &  & $4$ & $3.2$  & $1.0$  & $22.9$  & $1.0$  & $15.8$  & $1.0$  \\
wine & $178$ & $2$ & $11.5$  & $1.0$  & $42.8$  & $1.0$  & $47.2$  & $1.0$  \\
 &  & $7$\tnote{1} & $4.3$ & $1.0$ & $23.3\; (1.4)$ & $1.0$ & $23.1 \; (1.9)$ & $1.0$ \\
gr202 & $202$ & $6$\tnote{1} & $34.1 \; (1.5)$ & $1.0$ & $82.5 \; (12.7)$ & $1.0$ & $232.2 \; (31.5)$ & $1.0$ \\
seeds & $210$ & $2$\tnote{1} & $17.6 \; (0.2)$ & $1.0$ & $110.0 \; (3.6)$ & $1.0$ & $140.6 \; (11.3)$ & $1.0$ \\
 &  & $3$ & $10.2$  & $1.0$  & $61.2$  & $1.0$  & $143.4$  & $1.0$ \\
glass & $214$ & $3$\tnote{\dag} & $119.3$  & $3.0$  & $426.7$  & $2.0$  & $421.9$  & $1.0$  \\
 &  & $6$\tnote{1}\tnote{\dag} & $97.1\; (21.2)$ & $3.2 \;(0.5)$ & $541.4 \;(78.7)$ & $4.0 \;(1.4)$ & $952.7 \;(47.0)$ & $3.0$ \\
CatsDogs & $328$ & $2$\tnote{1} & $72.3 \;(31.0)$ & $1.0$ & $780.7 \;(328.7)$ & $1.2 \;(0.5)$ & $1844.3 \;(153.6)$ & $1.00$ \\
accent & $329$ & $2$\tnote{*} & $88.2 \; (6.4)$ & $1.0$ & $707.4 \; (52.0)$ & $1.0$ & $1617.1 \; (183.7)$ & $1.0$ \\
 &  & $6$\tnote{1} & $78.7 \; (19.5)$ & $1.0$ & $512.7 \; (36.5)$ & $1.0$ & $8443.7 \; (407.7)$ & $9.2 \;(0.5)$ \\
ecoli & $336$ & $3$ & $339.8 \; (267.8)$ & $1.8 \;(0.8)$ & $217.8 \; (19.2)$ & $1.0$ & $3202.2 \; (232.0)$ & $1.0$ \\
RealEstate & $414$ & $3$ & $73.2 \;(2.8)$ & $1.0$ & $641.7 \;(65.6)$ & $1.0$ & $>14400$ & N/A \\
 &  & $5$ & $130.3 \; (21.5)$ & $1.0$ & $1308.3 \; (348.6)$ & $1.8 \; (0.5)$ & $> 14400$ & N/A \\ 
Wholesale & $440$ & $5$\tnote{1} & $141.1 \; (9.4)$ & $1.0$ & $717.7 \; (23.8)$ & $1.0$ & $> 14400$ & N/A \\
 &  & $6$\tnote{1}\tnote{\dag} & $1106.6 \; (276.3)$ & $5.6 \; (2.1)$ & $4854.6 \; (1627.0)$ & $6.8 \; (0.8)$ & $ >14400$ & N/A \\ 
 \hline
\end{tabular}
\begin{tablenotes}
\item[1] {\tt k-means++}'s solution is not globally optimal for all five random trials.
\item[\dag] Problem~\eqref{lp:pK}  with $t=2$ is not tight.
\item[*] The SDP relaxation, \ie Problem~\eqref{SDP} is tight.
\end{tablenotes}
\end{threeparttable}
\end{table}

Next we consider larger data sets. We consider first-order LP solvers only, as {\tt Gurobi} is unable to solve any of the larger instances. Results are shown in Table~\ref{tab:bigdata}. While {\tt cuPDLP-C} solves all $32$ instances, {\tt PDLP} solves only $24$ instances within the time limit. Moreover, for the latter $24$ instances, {\tt cuPDLP-C} is on average $11.6$ times faster than {\tt PDLP}. \emph{Again, for all $32$ instances, the LP relaxation is tight.}  On the contrary, the SDP relaxation is tight only for $5$ instances. Moreover, the {\tt k-means++}'s solution is optimal for $22$ instances.

\begin{table}[htb]
\caption{Average run time and average number of iterations of the cutting plane algorithm  with {\tt PDLP} and {\tt cuPDLP-C} for large data sets. The numbers in parentheses denote the standard deviations across five random trials. No parentheses implies identical results for all trials.}
\label{tab:bigdata}
\small
\centering
\begin{threeparttable}
\begin{tabular}{|ccc|cc|cc|}
\hline
\multirow{2}{*}{Data set} & \multirow{2}{*}{$n$} & \multirow{2}{*}{$K$} & \multicolumn{2}{c|}{\tt cuPDLP-C} & \multicolumn{2}{c|}{\tt PDLP} \\ \cline{4-7} 
 &  &  & time(s) & iteration & time(s) & \multicolumn{1}{c|}{iteration} \\ \hline
ECG5000 & $500$ & $2$ & $120.7\; (15.7)$ & $1.0$ &  $484.3 \;(43.6)$ & $1.0$  \\
 &  & $5$\tnote{1}\tnote{\dag} & $581.7 \; (93.8)$ & $3.2 \; (0.5)$ &  $3727.6 \; (1115.2)$ & $5.4 \; (1.1)$  \\
Hungarian & $522$ & $2$ & $96.9 \; (11.3)$ & $1.0$ &   $599.8 \; (109.4)$ & $1.0$   \\
Wdbc & $569$ & $2$ & $88.9 \; (12.2)$ & $1.0$ &   $218.8 \; (457.7)$ & $1.6 \; (0.6)$  \\
 &  & $5$\tnote{1}  & $23.4 \;(20.1)$ & $1.0$ &  $1448.6 \; (437.7)$ & $1.8 \; (0.5)$ \\
Control & $600$ & $3$ & $93.9 \; (7.0)$ & $1.0$ &   $532.3 \; (41.6)$ & $1.0$ \\
Heartbeat & $606$ & $2$\tnote{*} & $158.1 \; (39.6)$ & $1.0$ &  $712.8 \; (31.4)$ & $1.0$ \\
Strawberry & $613$ & $2$ & $94.7 \; (5.5)$ & $1.0$   & $1039.7 \; (525.4)$ & $1.4 \; (0.6)$  \\
Energy & $768$ & $2$\tnote{*} & $50.3 \; (0.4)$ & $1.0$   & $169.5 \; (13.2)$ & $1.0$  \\
 &  & $12$\tnote{1}\tnote{*} & $158.5 \; (91.2)$ & $1.8 \; (1.1)$ & $869.5 \; (514.7)$ & $2.8 \; (2.1)$  \\
Gene & $801$ & $5$\tnote{1} & $211.2 \; (104.3)$ & $1.2 \; (0.5)$ &  $1104.7 \; (414.6)$ & $1.4 \; (0.6)$  \\
 &  & $6$\tnote{1} & $707.4 \; (111.2)$ & $3.0$ & $974.1 \; (340.4)$ & $1.2 \; (0.5)$  \\
SalesWeekly & $811$ & $2$\tnote{1} & $70.1 \; (3.9)$ & $1.0$ &  $1620.7 \; (118.1)$ & $2.0$  \\
 &  & $3$ & $65.3 \; (2.1)$ & $1.0$   & $565.1 \;(50.10)$ & $1.0$ \\
 &  & $5$ & $89.9 \; (2.7)$ & $1.0$ &  $1417.4 \; (433.1)$ & $1.8 \; (0.5)$ \\
Vehicle & $846$ & $2$ & $61.3 \; (1.5)$ & $1.0$ &  $866.8 \; (326.1)$ & $1.2 \; (0.5)$  \\
Arcene & $900$ & $2$ & $115.5 \; (5.7)$ & $1.0$ &  $273.6 \; (15.0)$ & $1.0$ \\
 &  & $3$\tnote{*} & $74.6 \; (3.4)$ & $1.0$ &  $292.1 \; (8.1)$ & $1.0$ \\
 &  & $5$\tnote{1}\tnote{\dag} & $759.7 \; (78.7)$ & $3.8 \; (0.5)$  & $ > 14400$ & N/A \\
Wafer & $1000$ & $2$ & $67.4 \; (1.5)$ & $1.0$ &  $316.9 \; (32.1)$ & $1.0$ \\
 &  & $4$\tnote{1}\tnote{\dag} & $643.1 \; (76.3)$ & $3.6 \; (0.6)$  & $4602.1 \; (1101.9)$ & $5.2 \; (1.1)$  \\
Power & $1096$ & $2$ & $95.9 \; (19.9)$ & $1.0$ &  $1543.2 \; (508.9)$ & $2.2 \; (0.8)$  \\
Phising & $1353$ & $9$\tnote{1}\tnote{\dag}& $840.5 \; (47.2)$ & $20.4 \; (1.8)$ &  $> 14400$ & N/A  \\
Aspirin & $1500$ & $3$ & $230.8 \; (119.1)$ & $1.4 \;(0.6)$ &  $1651.6 \;(894.8)$ & $2.4 \;(1.5)$  \\
Car & $1727$ & $4$ & $1504.6 \;(198.9)$ & $17.4 \;(1.1)$ & $10587.8 \;(1872.9)$ & $18.8 \;(1.6)$  \\
Ethanol & $2000$ & 2\tnote{*} & $210.2 \;(4.1)$ & $1.0$ &  $10330.9 \;(2217.6)$ & $7.0 \;(0.7)$  \\
Wifi & $2000$ & $5$\tnote{\ddag}& $3309.2 \;(702.6)$ & $8.0 \;(1.6)$ & $>14400$ & N/A  \\
Mallat & $2400$ & $3$\tnote{\ddag}& $5934.9 \; (1110.8)$ & $10.6 \;(1.5)$ &  $> 14400$ & N/A  \\
 &  & $4$\tnote{\ddag} & $7820.3 \;(1941.1)$ & $15.2 \;(2.2)$ & $> 14400$ & N/A  \\
Advertising & $3279$ & $2$\tnote{\ddag} & $4318.1\; (622.6)$ & $8.4\; (2.1)$ & $> 14400$ & N/A  \\
 &  & $8$\tnote{1}\tnote{\ddag} &  $7410.7\; (2367.0)$ & $12.0 \; (2.9)$  & $> 14400$ & N/A \\ 
 Rice & $3810$ & $2$\tnote{\ddag} & $8314.7\; (669.2)$ & $11.8\; (0.5)$ & $> 14400$ & N/A  \\ \hline
\end{tabular}

\begin{tablenotes}
\item[1] {\tt k-means++}'s solution is not globally optimal for all five random trials.
\item[\dag] Problem~\ref{lp:pK} with $t=2$ is not tight.
\item[\ddag] Unable to verify whether Problem~\eqref{lp:pK} with $t=2$ is tight or not due to exceeding the time limit.
\item[*] The SDP relaxation, \ie Problem~\eqref{SDP} is tight.
\end{tablenotes}
\end{threeparttable}
\end{table}

In Tables~\ref{tab:smalldata} and~\ref{tab:bigdata}, all instances for which Problem~\eqref{lp:pK} with $t=2$ is not tight or is not solvable within $4$ hours are marked. As we mentioned before, by letting $t > 2$, all such instances are solved to global optimality within the time limit, indicating the importance of selecting $t$ in a dynamic fashion. Recall that by construction, if $K=2$ or if an instance is solved in one iteration, then we have $t=2$.
For the remaining problems, in Table~\ref{tab:distt}, we show the distribution of $t$ associated with all inequalities added to the LP in the course of~\cut. As can be seen from this table, while for the majority of the inequalities we have $t\in \{2,3\}$, for some instances adding inequalities with $t \in \{4,5,6\}$ is beneficial as well.

\begin{table}[htb]
\caption{Distribution of different types of inequalities~\eqref{eq3k} (in percentage) added to the LP relaxation during cut generation with {\tt cuPDLP-C} as the LP solver. Inequalities~\eqref{eq3k} with $t=2$ added during the initialization step are not included.}
\label{tab:distt}
\small
\centering
\begin{tabular}{|ccc|ccccc|}
\hline
Data set & $n$ & $K$ & $t=2$ & $t=3$ & $t=4$ & $t=5$ & $t=6$ \\ \hline
glass & $214$ & $3$ & $13.2$ & $86.8$ & $0.0$ & $0.0$ &
$0.0$\\
 &  & $6$ &  $67.5$ & $32.5$ & $0.0$ &$0.0$  &$0.0$  \\
ecoli & $336$ & $3$ & $80.6$ & $19.4$ & $0.0$ & $0.0$ &  $0.0$\\
Wholesale & $440$ & $6$ & $51.0$ & $41.5$ & $6.2$ & $1.2$ & $0.1$ \\
ECG5000 & $500$ & $5$ & $53.7$ & $46.3$ &$0.0$  &$0.0$  &$0.0$  \\
Energy & $768$ & $12$ & $100.0$ &$0.0$  &$0.0$  & $0.0$ & $0.0$ \\
Gene & $801$ & $5$ & $100.0$ &$0.0$  &$0.0$  & $0.0$ &$0.0$  \\
 &  & $6$  & $43.5$ & $56.5$ & $0.0$ & $0.0$ & $0.0$ \\
Arcene & $900$ & $5$ & $37.4$ & $62.6$ & $0.0$ & $0.0$ &
$0.0$\\
Wafer & $1000$ & $4$ & $46.5$ & $53.5$ &$0.0$  &$0.0$  &
$0.0$\\
Phising & $1353$ & $9$ & $90.5$ & $9.5$ &$0.0$  &$0.0$  & $0.0$ \\
Aspirin & $1500$ & $3$ & $100.0$ & $0.0$ & $0.0$ & $0.0$ & $0.0$ \\
Car & $1727$ & $4$ & $77.9$ & $17.7$ & $4.4$ &$0.0$  &
$0.0$\\
Wifi & $2000$ & $5$ & $37.5$ & $49.2$ & $10.8$ & $2.5$ & 
$0.0$\\
Mallat & $2400$ & $3$ & $32.2$ & $67.8$ & $0.0$ &$0.0$  & $0.0$ \\
 &  & $4$ & $37.2$ & $40.3$ & $22.5$ &$0.0$  &$0.0$  \\
Advertising & $3279$ & $8$ & $84.9$ & $12.1$ & $2.9$ & $0.1$  & $0.0$ \\ \hline
\end{tabular}
\end{table}

We conclude this paper by acknowledging that further work is needed to make~\cut relevant for larger data sets. Of course, as first-order LP solvers become faster and more powerful GPUs become available, one can solve larger instances. Exploring sketching techniques~\cite{mixxie21,abdBan22} for tackling larger data sets is an interesting topic of future research as well. An alternative strategy is to incorporate the proposed LP relaxation within mixed-integer nonlinear programming solvers~\cite{IdaNick18} to expedite their convergence.

\clearpage

\begin{footnotesize} 
\bibliographystyle{plainurl}
\bibliography{biblio}

\begin{thebibliography}{10}

\bibitem{mosek}
{\em MOSEK 9.2}, 2019.
\newblock \url{http://docs.mosek.com/9.0/faq.pdf}.

\bibitem{abdBan22}
P.~Abdalla and A.~S. Bandeira.
\newblock Community detection with a subsampled semidefinite program.
\newblock {\em Sampling Theory, Signal Processing, and Data Analysis}, 20(1):6,
  2022.

\bibitem{Aloise09}
D.~Aloise, A.~Deshpande, P.~Hansen, and P.~Popat.
\newblock {NP}-hardness of {E}uclidean sum-of-squares clustering.
\newblock {\em Machine Learning}, 75:245--248, 2009.

\bibitem{applegate2021practical}
D.~Applegate, M.~D{\'\i}az, O.~Hinder, H.~Lu, M.~Lubin, B.~O'Donoghue, and
  W.~Schudy.
\newblock Practical large-scale linear programming using primal-dual hybrid
  gradient.
\newblock {\em Advances in Neural Information Processing Systems},
  34:20243--20257, 2021.

\bibitem{Awaetal15}
P.~Awasthi, A.~S. Bandeira, M.~Charikar, R.~Krishnaswamy, S.~Villar, and
  R.~Ward.
\newblock Relax, no need to round: Integrality of clustering formulations.
\newblock {\em Proceedings of the 2015 Conference on Innovations in Theoretical
  Computer Science}, 165:191--200, 2015.

\bibitem{neos98}
J.~Czyzyk, M.~P. Mesnier, and J.~J. More.
\newblock The {NEOS} server.
\newblock {\em IEEE Journal on Computational Science and Engineering}, 5(3):68
  --75, 1998.

\bibitem{openmp}
L.~Dagum and R.~Menon.
\newblock {OpenMP}: an industry standard api for shared-memory programming.
\newblock {\em IEEE computational science and engineering}, 5(1):46--55, 1998.

\bibitem{AntoAida20}
A.~De~Rosa and A.~Khajavirad.
\newblock The ratio-cut polytope and {K}-means clustering.
\newblock {\em SIAM Journal on Optimization}, 32(1):173--203, 2022.

\bibitem{dua2017uci}
D.~Dua, C.~Graff, et~al.
\newblock {UCI} machine learning repository.
\newblock {\em URL http://archive. ics. uci. edu/ml}, 2017.

\bibitem{FriRezSal19}
Z.~Friggstad, M.~Rezapour, and Salavatipour~M. R.
\newblock Local search yields a {PTAS} for {K}-means in doubling metrics.
\newblock {\em SIAM Journal on Computing}, 48(2):452–480, 2019.

\bibitem{gurobi}
{Gurobi Optimization, LLC}.
\newblock {Gurobi Optimizer Reference Manual}, 2021.
\newblock URL: \url{https://www.gurobi.com}.

\bibitem{IduMixPetVil17}
T.~Iguchi, D.~G. Mixon, J.~Peterson, and S.~Villar.
\newblock Probably certifiably correct {K}-means clustering.
\newblock {\em Mathematical Programming}, 165:605--642, 2017.

\bibitem{Kan02}
T.~Kanungo, D.M. Mount, N.S. Netanyahu, C.D. Piatko, R.~Silverman, and A.Y. Wu.
\newblock A local search approximation algorithm for {K}-means clustering.
\newblock {\em Proceedings of the 18th Annual ACM Symposium on Computational
  Geometry}, pages 10--18, 2002.

\bibitem{IdaNick18}
A.~Khajavirad and N.~V. Sahinidis.
\newblock A hybrid {LP/NLP} paradigm for global optimization relaxations.
\newblock {\em Mathematical Programming Computation}, 10:383--421, 2018.
\newblock \href {https://doi.org/10.1007/s12532-018-0138-5}
  {\path{doi:10.1007/s12532-018-0138-5}}.

\bibitem{LiLiLi20}
X.~Li, Y.~Li, S.~Ling, T.~Strohmer, and K.~Wei.
\newblock When do birds of a feather flock together? {K}-means, proximity, and
  conic programming.
\newblock {\em Mathematical Programming}, 179:295--341, 2020.

\bibitem{Lloyd82}
S.~Lloyd.
\newblock Least squares quantization in {PCM}.
\newblock {\em IEEE Transactions on Information Theory}, 28(2):129 --137, 1982.

\bibitem{lu2023cupdlp}
H.~Lu and J.~Yang.
\newblock {cuPDLP.jl}: A {GPU} implementation of restarted primal-dual hybrid
  gradient for linear programming in julia.
\newblock {\em arXiv preprint arXiv:2311.12180}, 2023.

\bibitem{lu2023cupdlpC}
H.~Lu, J.~Yang, H.~Hu, Q.~Huangfu, J.~Liu, T.~Liu, Y.~Ye, C.~Zhang, and D.~Ge.
\newblock {cuPDLP-C}: A strengthened implementation of cupdlp for linear
  programming by c language.
\newblock {\em arXiv preprint arXiv:2312.14832}, 2023.

\bibitem{MahNimVar09}
M.~Mahajan, P.~Nimbhorkar, and K.~Varadarajan.
\newblock The planar {K}-means problem is {NP}-hard.
\newblock In {\em WALCOM: Algorithms and Computation}, pages 274--285. Springer
  Berlin Heidelberg, 2009.

\bibitem{Man79}
O.L. Mangasarian.
\newblock {Uniqueness of solution in linear programming}.
\newblock {\em Linear Algebra and its Applications}, 25:151--162, 1979.

\bibitem{marzi2019computational}
F.~Marzi, F.~Rossi, and S.~Smriglio.
\newblock Computational study of separation algorithms for clique inequalities.
\newblock {\em Soft Computing}, 23(9):3013--3027, 2019.

\bibitem{MixVilWar17}
D.~G. Mixon, S.~Villar, and R.~Ward.
\newblock Clustering subgaussian mixtures by semidefinite programming.
\newblock {\em Information and Inference: A Journal of the IMA}, 6(4):389--415,
  2017.

\bibitem{mixxie21}
D.~G. Mixon and K.~Xie.
\newblock Sketching semidefinite programs for faster clustering.
\newblock {\em IEEE Transactions on Information Theory}, 67(10):6832--6840,
  2021.

\bibitem{neumaier2004safe}
A.~Neumaier and O.~Shcherbina.
\newblock Safe bounds in linear and mixed-integer linear programming.
\newblock {\em Mathematical Programming}, 99:283--296, 2004.

\bibitem{Pad89}
M.~Padberg.
\newblock The boolean quadric polytope: Some characteristics, facets and
  relatives.
\newblock {\em Mathematical Programming}, 45:139--172, 1989.

\bibitem{PenWei07}
J~Peng and Y~Wei.
\newblock Approximating {K}-means-type clustering via semidefinite programming.
\newblock {\em SIAM Journal on Optimization}, 18(1):186--205, 2007.

\bibitem{PenXia05}
J.~Peng and Y.~Xia.
\newblock {\em A New Theoretical Framework for {K}-Means-Type Clustering},
  pages 79--96.
\newblock Springer Berlin Heidelberg, 2005.

\bibitem{PicSudWie22}
V.~Piccialli, A.~M. Sudoso, and A.~Wiegele.
\newblock {SOS-SDP}: An exact solver for minimum sum-of-squares clustering.
\newblock {\em INFORMS Journal on Computing}, 34(4):2144--2162, 2022.

\bibitem{PraHan18}
M.~N. Prasad and G.~A. Hanasusanto.
\newblock Improved conic reformulations for {K}-means clustering.
\newblock {\em SIAM Journal on Optimization}, 28(4):3105--3126, 2018.

\bibitem{VerBookHDP}
R.~Vershynin.
\newblock {\em High-dimensional probability: An introduction with applications
  in data science}, volume~47.
\newblock Cambridge University Press, 2018.

\end{thebibliography}
\end{footnotesize}

\clearpage

\section{Appendix: Proof of Lemma~\ref{MasterOfProbability0}}
\label{sec:appendix}

\begin{proof}
We first prove the following:
\begin{claim}\label{claim1}
    Inequalities~\eqref{22221}-\eqref{22220} can be equivalently written as:
\begin{equation}\label{0000}
\max_{x,y \in \partial \B_1}r\dashint_{\partial \B_2} \max\{x^Tz, y^Tz\} d\H^{m-1}(z) -x^Ty< \frac {r}2\Delta^2,\quad \mbox{for $r=r_1,r_2$}.
\end{equation}
\end{claim}
\begin{cpf}
As before, for notational simplicity, we denote the $i$th (resp. $j$th, $k$th) point by $x$ (resp. $y$, $z$). 
By \eqref{0} and \eqref{p}, inequalities~\eqref{22221}-\eqref{22220} read respectively
$$
\min_{x,y \in \partial \B_2}r_1\dashint_{\partial \B_1} \min\{\|x-z\|^2, \|y-z\|^2\} d\H^{m-1}(z) - \|x-y\|^2>   2(r_1-1),
$$
$$
\min_{x,y \in \partial \B_1}r_2\dashint_{\partial \B_2} \min\{\|x-z\|^2, \|y-z\|^2\} d\H^{m-1}(z) - \|x-y\|^2>   2(r_2-1).
$$
The first inequality, up to a change of variable, reads:
$$
\min_{x,y \in \partial \B_1}r_1\dashint_{\partial \B_2} \min\{\|x-z\|^2, \|y-z\|^2\} d\H^{m-1}(z) - \|x-y\|^2>   2(r_1-1),
$$
hence inequalities~\eqref{22221}-\eqref{22220} read
$$
\min_{x,y \in \partial \B_1}r\dashint_{\partial \B_2} \min\{\|x-z\|^2, \|y-z\|^2\} d\H^{m-1}(z) - \|x-y\|^2>   2(r-1),\quad \mbox{for $r=r_1,r_2$}
$$
which, expanding the squares, gives
\begin{equation}\label{000}
\min_{x,y \in \partial \B_1}r\dashint_{\partial \B_2}\|z\|^2+ 1+\min\{-2x^Tz, -2y^Tz\} d\H^{m-1}(z) +2x^Ty -2>2(r-1),\quad \mbox{for $r=r_1,r_2$}.
\end{equation}
Via a change of variables, we have
\begin{equation*}
\begin{split}
\dashint_{\partial \B_2}\|z\|^2&d\H^{m-1}(z) = \dashint_{\partial \B_1}\|\Delta e_1 +z\|^2d\H^{m-1}(z)\\
&=\dashint_{\partial \B_1}\|\Delta e_1\|^2+\|z\|^2+2\Delta e_1^Tzd\H^{m-1}(z)=\Delta^2+1.
\end{split}
\end{equation*}
Hence \eqref{000} reads
\begin{equation*}
\min_{x,y \in \partial \B_1}r\dashint_{\partial \B_2} \min\{-2x^Tz, -2y^Tz\} d\H^{m-1}(z) +2x^Ty> - r\Delta^2,\quad \mbox{for $r=r_1,r_2$},
\end{equation*}
which is equivalent to \eqref{0000}.
\end{cpf}
Thanks to Claim \ref{claim1}, to conclude the proof of the lemma it suffices to show the following:
\begin{claim}\label{l:condition}
Inequalities \eqref{0000} hold if and only if  $\Delta > \Delta_0$.
\end{claim}
\begin{cpf}
We will prove that the maximum of the left-hand side of inequalities~\eqref{0000} over all $x, y \in \partial \B_1$ is attained at $(e_1,-e_1)$, for all $r\in (0,2)$.
Since $r_1\in (0,1]$ and  $r_2\in [1,2)$, this in turn implies that inequalities~\eqref{0000} are satisfied if and only if
$$
r\Delta+1=r\dashint_{\partial \B_2} z^Te_1 d\H^{m-1}(z) +1< \frac {r}2\Delta^2,\quad \mbox{for $r=r_1,r_2$},
$$
which, by $r_1\leq r_2$, is true if and only if $\Delta > \Delta_0$; \ie the desired condition.
Define
\begin{equation}\label{defrr}
F_r(x,y):=r\dashint_{\partial \B_2} \max\{x^Tz, y^Tz\} d\H^{m-1}(z) -x^Ty.    
\end{equation}
Our goal is to show that for every $r\in (0,2)$
\begin{equation}\label{p10}
 \max_{x,y\in \partial \B_1} F_r(x,y) = F_r(e_1,-e_1)=r\Delta+1.
\end{equation}
In Lemma \ref{l:inter}, we prove that
\begin{equation}\label{claim}
 \max_{x,y\in \partial \B_1} F_2(x,y) = F_2(e_1,-e_1)=2\Delta+1.
\end{equation}
Hence, recalling that $r\in (0,2)$ and that $1-\frac r2\geq -(1-\frac r2)x^Ty$ for every $x,y\in \partial \B_1$, we conclude with the following chain of inequalities:
\begin{equation}
    \begin{split}
r\Delta+1&=\frac r2(2\Delta+1)+(1-\frac r2)=\frac r2\max_{x,y\in \partial \B_1} F_2(x,y)+(1-\frac r2)\\
&= \max_{x,y\in \partial \B_1} \frac r2F_2(x,y)+(1-\frac r2)\geq \max_{x,y\in \partial \B_1} F_r(x,y).
        \end{split}
\end{equation}

\end{cpf}
\end{proof}

In order to prove Lemma \ref{MasterOfProbability0}, we made use of the next lemma, for which we provide a proof that is closely related to the proof of Lemma~1 in~\cite{AntoAida20}, in which the authors prove
$$
 \max_{x,y\in \B_1} F_1(x,y) = F_1(e_1,-e_1)=\Delta+1,
$$
where $F_r$ is defined by~\eqref{defrr} and $\B_1$ denotes a ball of radius one as defined in the SBM. For brevity, in the following, we only include the parts of the proof that are different, and when possible, we refer to the relevant parts of the proof of Lemma~1 in~\cite{AntoAida20}.  

\begin{lemma}\label{l:inter}
    \eqref{claim} holds.
\end{lemma}
\begin{proof}
For any $x\in \R^m$, we denote by $x_i$ the $i$th component of $x$. 
We divide the proof in several steps:

\smallskip

\noindent
{\em {\bf Step 1}. Slicing:

Let $z, w$ be any pair of points in $\partial \B_2$ satisfying $z_1=w_1$, $z_2=-w_2\geq 0$,  $z_j=w_j=0$ for all $j \in \{3,\ldots,m\}$. In the special case $m=1$, we consider $z=w\in \partial \B_2$.
 Define
$$
H(x,y):= \max\{x^Tz,y^Tz\} +  \max\{x^Tw,y^Tw\} -x^Ty .
$$
Then \eqref{claim} holds if the following holds 
\begin{equation}\label{p20}
\max_{x,y\in \partial \B_1}H(x,y)=H(e_1,-e_1)
\end{equation}
for any pair of points $z, w\in\partial \B_2$ satisfying $z_1=w_1$, $z_2=-w_2\geq 0$,  $z_j=w_j=0$ for all $j \in \{3,\ldots,m\}$ in case $m\geq 2$, or, in case $m=1$, for any point $z=w\in \partial \B_2$ .
}

{\em Proof of Step 1.} Since
$$F_2(x,y)=\frac 1{\H^{m-1}(\partial B_2)}\int_{\Delta-1}^{\Delta+1} \int_{\{z_1=s\}\cap \partial B_2}2\max\{x^Tz,y^Tz\} -x^Ty d\H^{m-2}(z)ds,$$
to show \eqref{claim} it is enough to show that the function
$$G(x,y):= \int_{\{z_1=s\}\cap \partial B_2}2\max\{x^Tz,y^Tz\} -x^Ty d\H^{m-2}(z)$$
is maximized in $x=e_1$, $y=-e_1$, for every $s \in [\Delta-1,\Delta+1].$
Denoting $A:=\{z_1=s,\, z_2\geq 0\}\cap \partial B_2$, then
$$\frac 12 G(x,y)= \int_{A}\max\{x^Tz,y^Tz\}  + \max\{x^T(2se_1-z),y^T(2se_1-z)\} -x^Ty d\H^{m-1}(z).$$
Hence, it is enough to prove that for every $s \in [\Delta-1,\Delta+1]$ and for every $z\in A$,
\begin{equation}\label{p30}
\max_{x,y\in \partial B_1}\left \{ \max\{x^Tz,y^Tz\} +  \max\{x^T(2se_1-z),y^T(2se_1-z)\} -x^Ty \right\},
 \end{equation}
 is achieved at $(e_1,-e_1)$.
 Since Problem \eqref{p30} is invariant under a rotation of the space around the axis generated by $e_1$,  we conclude that solving Problem \eqref{p30} is equivalent to solving Problem \eqref{p20}.

\noindent
{\em {\bf Step 2}. Symmetric distribution of the maxima: 

Let $z,w$ be any pair of points as defined in Step 1.
Define
$$I(x,y):= x^Tz + y^Tw -x^Ty.$$
In order to show that~\eqref{p20} holds, it suffices to prove that for $m\geq 2$
\begin{equation}\label{p40}
\mathop{\max_{x,y\in \partial B_1}}_{x^Tz\geq y^Tz,\, y^Tw \geq x^Tw }\{I(x,y) \} \leq H(e_1,-e_1)=2z_1+1.
\end{equation}
}
{\em Proof of Step 2.}  The proof is identical (up to trivial changes) to the proof of Step 2 of Lemma 1 in \cite{AntoAida20}.

\noindent
{\em {\bf Step 3}. Reduction from spheres to circles:

To show the validity of~\eqref{p40}, we can restrict to dimension $m=2$.}

{\em Proof of Step 3.}  The proof repeats verbatim as in the proof of Step 3 of Lemma 1 in \cite{AntoAida20}.

\noindent
{\em {\bf Step 4}. Symmetric local maxima:

For any pair $x,y\in \partial \B_1$ of the form $x_1 = y_1$ and $x_2 = -y_2$, we have
$$I(x,y)\leq H(e_1,-e_1).$$}
{\em Proof of Step 4.}
Given such symmetric pair $(x,y)$, the objective function evaluates to
$I(x,y)= 2z_1 x_1 + 2z_2 x_2 -x_1^2 + x_2^2$.  Using $x^2_1 + x^2_2 = 1$ and  $z_2 \sqrt{1-x^2_1} \leq z_2$, it suffices to show that
\begin{equation}\label{eqA1}
2z_2 \leq  2 x_1^2 - 2z_1 x_1 + 2z_1, \quad \forall  x_1 \in [-1,1].
\end{equation}
Since the function $\hat f(x_1):=2 x_1^2 - 2z_1 x_1 + 2z_1$ on the right hand side of~\eqref{eqA1} is a convex parabola in $x_1$,
its minimum is either attained at one of the end points or at $\tilde x_1 = \frac{z_1}{2}$, provided
that $-1 \leq \frac{z_1}{2} \leq 1$. Since $\Delta -1 \leq z_1 \leq \Delta + 1$, the point $\tilde x_1$ lies in the domain only if
$\Delta-1 \leq z_1 \leq \min\{2, \Delta + 1\}=2$.
The value of $\hat f$ at $x_1 = -1$ and $x_1 = 1$ evaluates to $2 + 4 z_1$ and $2$, respectively,
both of which are bigger than $2z_2$. Hence it remains to show that $\hat f(\tilde x_1)\geq 2z_2$ if $\Delta-1 \leq z_1 \leq 2$, that is we have to show that
$$
2\sqrt{1- u^2} \leq 2(u + \Delta) - \frac{(u+\Delta)^2}{2}, \quad \mbox{if} \quad -1 \leq u \leq 2-\Delta,
$$
where we set $u := z_1 - \Delta$ and we use that $z_2= \sqrt{1-u^2}$. Since $u + \Delta \leq 2$, the right hand side of the above inequality is increasing in $\Delta$;
hence it suffices to show its validity at $\Delta = 2$; \ie
$$
2\sqrt{1- u^2} \leq 2(u + 2) - \frac{(u + 2)^2}{2}, \quad \mbox{if} \quad  -1 \leq u \leq 0,
$$
which can be easily proved.

\noindent
{\em {\bf Step 5}. Decomposition of the circle:

To solve Problem~\eqref{p40}, it suffices to solve
\begin{equation}\label{p6}
\mathop{\max_{x,y\in  \partial \B_1\cap \{x_1\leq 0,\, y_1 \geq 0\}}}_{x^Tz\geq y^Tz, \, y^Tw \geq x^Tw }I(x,y) \leq H(e_1,-e_1)=2z_1+1,
\end{equation}}
for every $z\in \partial \B_2$, $z_1=w_1$, $z_2=-w_2\geq 0$.

{\em Proof of Step 5.}  The proof is identical (up to trivial changes) to the proof of Step 6 of Lemma 1 in \cite{AntoAida20}.

\noindent
{\em {\bf Step 6}. We solve Problem \eqref{p6}.}

{\em Proof of Step 6.}  The proof is identical (up to trivial changes) to the proof of Step 7 of Lemma 1 in \cite{AntoAida20}.

\end{proof}

\end{document}